\documentclass{article}
\usepackage[utf8]{inputenc}
\usepackage{amsmath,amssymb,amstext,amsthm, geometry, verbatim, tikz-cd, enumerate}
\newtheorem*{theorem*}{Theorem}
\newtheorem*{proposition*}{Proposition}
\newtheorem*{corollary*}{Corollary}
\newtheorem{theorem}{Theorem}[section]
\newtheorem{lemma}[theorem]{Lemma}
\newtheorem{proposition}[theorem]{Proposition}
\newtheorem{corollary}{Corollary}[theorem]
\theoremstyle{definition}
\newtheorem{definition}{Definition}
\newtheorem{example}{Example}
\theoremstyle{remark}
\newtheorem{remark}{Remark}

\newcommand \vp{\varphi}

\newcommand \id{\text{Id}}

\newcommand \diff{\mathrm{d}}
\numberwithin{equation}{section}
\newcommand \mc{\mathcal}
\newcommand \mb{\mathbb}

\DeclareMathOperator{\Hom}{Hom}
\DeclareMathOperator{\End}{End}
\DeclareMathOperator{\Pic}{Pic}

\DeclareMathOperator{\Spec}{Spec}
\DeclareMathOperator{\rank}{rank}

\DeclareMathOperator{\supp}{supp}

\DeclareMathOperator{\Aut}{Aut}

\DeclareMathOperator{\Ext}{Ext}

\DeclareMathOperator{\coker}{coker}
\DeclareMathOperator{\tr}{Tr}
\DeclareMathOperator{\chern}{ch}

\DeclareMathOperator{\Image}{im}
\title{(Co)-Higgs Bundles on Non-K\"ahler Elliptic Surfaces}
\author{Eric Boulter\footnote{Department of Mathematics and Statistics, University of Saskatchewan, Saskatoon, Canada.  Email: eric.boulter@usask.ca. This author was partially supported by NSERC Canada Graduate Student Doctoral Award [CGSD3-534983-2019]} and Ruxandra Moraru\footnote{Department of Pure Mathematics, University of Waterloo, Waterloo, Canada. Email: moraru@uwaterloo.ca. This author was supported by NSERC Discovery Grant [DG50503-10798]}}

\begin{document}

\maketitle

\begin{abstract}
In this paper, we study Higgs and co-Higgs bundles on non-K\"ahler elliptic surfaces. We show, in particular, that non-trivial stable Higgs bundles only exist when the base of the elliptic fibration has genus at least two and use this existence result to give explicit topological conditions ensuring the smoothness of moduli spaces of stable rank-2 sheaves on such surfaces. We also show that non-trivial stable co-Higgs bundles only exist when the base of the elliptic fibration has genus 0, in which case the non-K\"ahler elliptic surface is a Hopf surface. We then given a complete description of non-trivial co-Higgs bundles in the rank 2 case; these non-trivial rank-2 co-Higgs bundles are examples of non-trivial holomorphic Poisson structures on $\mathbb{P}^1$-bundles over Hopf surfaces.
\end{abstract}

\section{Introduction}\label{sec1}

Let $X$ be a compact $n$-dimensional complex manifold endowed with a Gauduchon metric $g$ whose fundamental form is $\omega$, and let $V$ be a fixed holomorphic vector bundle on $X$. Consider a pair $(E,\phi)$ consisting of a holomorphic vector bundle $E$ on $X$ and a holomorphic section $\phi \in H^0(X,\End E \otimes V)$ such that $\phi \wedge \phi = 0$. The section $\phi$ is  called a {\em Higgs field}. The pair $(E,\phi)$ is then said to be {\em $g$-stable} if for any proper subsheaf $\mathcal{S} \subset E$ such that $\phi(\mathcal{S}) \subset \mathcal{S} \otimes V$ (that is, $\mc{S}$ is {\em $\phi$-invariant}), we have 
 \[ \frac{\deg_g(\mc{S})}{\rank \mc{S}} < \frac{\deg_g(E)}{\rank E},\]
 in which case $\phi$ is a {\em $g$-stable Higgs field}. Note that the degree of a torsion-free coherent sheaf $\mc{F}$ on a Gauduchon manifold $(X,g)$ is defined as
 \[ \deg_g(\mc{F}) = \frac{i}{2\pi}\int_X F_h \wedge \omega^{n-1},\]
 where $F_h$ is the curvature of the Chern connection of any Hermitian metric $h$ on $\det \mc{F}$ \cite{LuTe}.
 Moreover, the endomorphism bundle $\End E$ of $E$ naturally splits as $\End E \simeq  \mc{O}_X \oplus \End_0 E$,
 where the sections of $\mc{O}_X$ correspond to multiples of $\id_E$ and the sections of $\End_0 E$ are the trace-free endomorphisms of $E$. We thus have
 \[ H^0(X,\End E \otimes V) = H^0(X,V) \oplus H^0(X,\End_0 E \otimes V).\]
 Given that sections of $H^0(X,V)$ can easily be computed, we focus our analysis on the non-zero trace-free Higgs fields, that is, sections $\phi \in H^0(X,\End_0 E \otimes V)$ with $\phi \neq 0$ and $\phi \wedge \phi = 0$, which we call {\em non-trivial}. Furthermore, if $E$ is a line bundle, then $\End E = \mathcal{O}_X$ and $H^0(X,\End E \otimes V) = H^0(X,V)$ so that $H^0(X,\End_0 E \otimes V) = \{ 0 \}$. We therefore only study non-trivial {\em $V$-pairs} $(E,\phi: E \rightarrow E \otimes V)$ for holomorphic vector bundles $E$ of rank $\geq 2$.
 
In this paper, we are interested in the case $V = \mc{T}_X$, $\mc{T}^*_X$ or $K_X$, where $\mc{T}_X$ and $\mc{T}^*_X$ are, respectively, the holomorphic tangent and cotangent bundles of $X$, and $K_X$ is its canonical bundle. If $V = \mc{T}_X$, the pairs $(E,\phi)$ are called {\em co-Higgs bundles} and correspond to generalised holomorphic vector bundles over the complex manifold $X$ \cite{Gualtieri,Hit-B,Rayan}. Co-Higgs bundles also play an important role in holomorphic Poisson geometry as they provide higher-dimensional examples of compact holomorphic Poisson manifolds. To be precise, given a rank-2 co-Higgs bundle $(E,\phi)$ on $X$, if the Higgs field $\phi$ is non-zero and trace-free, it induces a non-trivial holomorphic Poisson structure on $\mathbb{P}(E)$  \cite{Polishchuk, Rayan}. More generally, if $E$ has rank $r 
 \geq 3$ and $\phi$ satisfies an integrability condition stronger than $\phi \wedge \phi = 0$, then $\phi$ also induces a holomorphic structure on $\mathbb{P}(E)$ \cite{Matviichuk}. Stable rank-2 co-Higgs bundles have been studied on curves \cite{Rayan, Rayan-P1} and on certain projective surfaces \cite{Rayan-P2, Vicente, VicenteColmenares_Moduli, Biswas-Rayan}. Whereas examples of strongly integrable co-Higgs bundles of rank $\geq 3$ exist on projective space \cite{Matviichuk}.  

On the other hand, if $V = \mc{T}^*_X$ or $K_X$, the pairs $(E,\phi)$ are called {\em Higgs bundles} or {\em Vafa-Witten pairs}, respectively. Higgs bundles were first introduced by Hitchin \cite{Hit-Higgs} on algebraic curves and generalised to higher dimensional varieties by Simpson \cite{Simpson}. They also correspond to a special class of solutions of the Kapustin-Witten equations on compact K\"ahler surfaces \cite{Tanaka}. Although Higgs bundles over curves have been extensively studied over the past thirty-five years, less is known in the higher dimensional case. Some results have nonetheless been obtained in higher dimensions by Biswas, Bottacin and Biswas-Mj-Verbitsky \cite{Biswas2, Biswas, Bottacin-Higgs,BiswasVerbitsky}, among others. Moreover, stable Vafa-Witten pairs correspond to solutions of the Vafa-Witten equations \cite{VafaWitten} via the Kobayashi-Hitchin correspondence \cite{LinThesis, GarciaPrada, Tanaka}. Their moduli have been studied over projective K3 surfaces \cite{TanakaThomasI, TanakaThomasII, Kool}, and explicit examples have been constructed over projective K3 surfaces \cite{TanakaThomasI} and properly elliptic surfaces \cite{MarchesanoMoraruSavelli}.
 
 In this article, we consider trace-free co-Higgs and Higgs bundles over non-K\"ahler elliptic surfaces, which correspond to non-trivial principal bundles over a complex curve $B$ with structure group an elliptic curve. We first prove that non-trivial stable co-Higgs bundles only exist if the genus of $B$ is zero, in which case $X$ is a Hopf surface, and non-trivial stable Higgs bundles only exist if the genus of $B$ is greater than or equal to 2. This is done in sections \ref{Existence} and \ref{elliptic}.
 Note that this is consistent with the case of curves: it is known that on a curve $C$ of genus $g(C)$, non-trivial stable Higgs bundles only exist if $g(C) \geq 2$ \cite{Hit-Higgs}, whereas non-trivial stable co-Higgs bundles only exist if $g(C)=0$ \cite{Rayan}. The result for non-K\"ahler elliptic surfaces is, however, a consequence of the following more general result:
 
 \begin{proposition*}
 Let $(X,g)$ be a Gauduchon manifold such that $\mc{T}_X$ (resp. $\mc{T}_X^*$) is given by an extension of vector bundles $$0 \rightarrow V_1 \stackrel{i}{\rightarrow}  \mc{T}_X \stackrel{p}{\rightarrow}  V_2 \rightarrow 0.
$$
If there are no non-trivial $g$-stable pairs $(E, \vp_1:E \to E\otimes V_1)$ or $(E, \vp_2:E \to E\otimes V_2)$, then $X$ does not admit non-trivial $g$-stable co-Higgs bundles (resp. Higgs bundles).
 \end{proposition*}

 In section \ref{Hopf}, we classify all non-trivial rank-2 co-Higgs bundles on Hopf surfaces. The results we obtain can be summarised as follows:
 
\begin{theorem*}
Let $X$ be a Hopf surface. 
\begin{enumerate}
    \item 
    For any integer $c_2\geq 0$, there is a non-trivial stable rank-2 co-Higgs bundle $(E,\phi)$ on $X$ such that $c_2(E)=c_2$, in which case $E$ is a filtrable bundle.
    \item 
    The moduli space $\mathcal{M}^{st}_{coH,0}(X)$ of non-trivial stable rank-2 co-Higgs bundles  on $X$ with second Chern class $c_2=0$ is a 5-dimensional variety with two disjoint components $\mathcal{Z}_1$ and $\mathcal{Z}_2$ given by
\[ \mathcal{Z}_1 \simeq \Pic(X) \times \Spec \left(\mathbb{C}[s,z, v,w]/(w(4zv - s^2) - 1)\right) \times \mathbb{P}^1\]
and
\[ \mathcal{Z}_2 \simeq \Pic(X) \times \mathbb{C}^3 \times \mathbb{P}^1.\]
Moreover, $\mathcal{M}^{st}_{coH,0}(X)$ is a codimension-1 subvariety of the moduli space of co-Higgs bundles on $\mathbb{P}^1$.
\end{enumerate}
\end{theorem*}
As a direct consequence of this theorem, we also have:
\begin{corollary*}
There exist non-trivial holomorphic Poisson structures on $\mathbb{P}(E)$ for filtrable rank-2 bundles $E$ on Hopf surfaces.
\end{corollary*}
A complete classification of which projective bundles $\mathbb{P}(E)$ admit holomorphic Poisson structures is given in Proposition \ref{hopf-reg}. These are examples of three-dimensional holomorphic Poisson manifolds.
 
Finally, we study rank-2 Higgs bundles on non-K\"ahler elliptic surfaces in section \ref{geq2}. All of the results in this section also correspond to results on Vafa-Witten pairs because, for these surfaces, there is a one-to-one correspondence between Higgs bundles and Vafa-Witten pairs. We show, in particular, that when the genus of the base curve is at least two and $\delta \in \Pic(X)$ and $c_2 \in \mathbb{Z}$ are chosen to be in the filtrable range, there always exist non-trivial stable rank-2 Higgs pairs $(E,\phi)$ with $\det(E)=\delta$ and $c_2(E)=c_2$. In the non-filtrable range, we relate the existence of non-trivial Higgs fields on a bundle $E$ to a line bundle determined by the ramification data of the spectral curve of $E$. Using this result, we obtain bounds on invariants of $E$ such that the existence of non-trivial Higgs pairs is guaranteed. 

By standard deformation theory results, there is a correspondence between non-trivial Vafa-Witten pairs $(E,\phi)$ with $E$ stable and the obstruction class $H^2(X,\End_0(E))$ to deformations of $E$  on a complex surface $X$, so we use our results from Section \ref{geq2} to make conclusions about the smoothness of moduli spaces of stable bundles on non-K\"ahler elliptic surfaces. Note that moduli spaces of stable bundles on non-K\"ahler elliptic surfaces were studied in \cite{BrMor2}, where the authors showed that the moduli spaces are always smooth when the base curve of the elliptic fibration has genus 0 or 1 \cite[Proposition 4.1]{BrMor2}. They also showed that the locus of regular bundles is smooth for a base curve of genus $\geq 2$ when $\Delta(2, \delta, c_2)>(g-1)/2$ \cite[Proposition 4.2]{BrMor2}. By also considering the invariant $e_\delta$ associated to the determinant $\delta$ of $E$ defined in Section \ref{Line bundles}, we can improve these results to the following:
\begin{theorem*}
    Let $\pi:X\to B$ be a non-K\"ahler principal elliptic surface with $g(B)=g\geq 2$ and let $\mc{M}_{2,\delta,c_2}(X)$ be the moduli space of stable rank-2 vector bundles on $X$ with determinant $\delta$ and second Chern class $c_2$.
    \begin{enumerate}
        \item If $\Delta(2,\delta,c_2)\geq m(2,\delta)$, then there are non-trivial stable Vafa-Witten pairs $(E,\phi)$ with $\rank(E)=2$, $\det(E)=\delta$, and $c_2(E)=c_2$, so $\mc{M}_{2,\delta,c_2}(X)$ is not smooth as a ringed space.
        \item If $4\Delta(2,\delta,c_2)=-e_\delta<g-1$, then there are non-trivial stable Vafa-Witten pairs $(E,\phi)$ with $\rank(E)=2$, $\det(E)=\delta$, and $c_2(E)=c_2$, so $\mc{M}_{2,\delta,c_2}(X)$ is not smooth as a ringed space.
        \item If $g\geq 3$ and $e_\delta>2-g$, then there are non-trivial stable Vafa-Witten pairs $(E,\phi)$ with $\rank(E)=2$, $\det(E)=\delta$, and $c_2(E)=c_2$, so $\mc{M}_{2,\delta,c_2}(X)$ is not smooth as a ringed space.
        \item If $g=2$, $e_\delta=-2$, and $1/2<\Delta(2,\delta,c_2)<m(2,\delta)$, then there are no non-trivial stable Vafa-Witten pairs $(E,\phi)$ with $\rank(E)=2$, $\det(E)=\delta$, and $c_2(E)=c_2$, so $\mc{M}_{2,\delta,c_2}(X)$ is smooth.
    \end{enumerate}
    (Note that there is overlap in the hypotheses of points 1-3.)
\end{theorem*}

We finish the paper by describing an algorithm to compute all Higgs bundles in the non-filtrable range on an elliptic surface with base of genus two, given an explicit description of the surface.

\section{A Necessary condition for trace-free Higgs fields}\label{Existence} 

Let $X$ be a compact $n$-dimensional complex manifold with Gauduchon metric $g$ and $V$ be a fixed holomorphic vector bundle on $X$. 

\subsection{Degree and stability}
The degree of a holomorphic vector bundle can be defined on any compact complex manifold $X$. Recall that a Hermitian metric $g$ on $X$ is called {\em Gauduchon} if its fundamental form $\omega$ satisfies $\partial \Bar{\partial} \omega^{n-1} = 0$, where $n = \dim_{\mathbb{C}} X$ \cite{Gauduchon}. Such metrics always exist on a compact complex manifold. Let $L$ be a holomorphic line bundle on $X$. The {\em degree of $L$ with respect to $g$} is defined by
\[ \deg_g L := \frac{i}{2\pi}\int_X F \wedge \omega^{n-1},\]
where $F$ is the curvature of a Hermitian connection on $L$ compatible with its holomorphic structure. Any two such forms differ by a form that is $\partial \Bar{\partial}$-exact. Nonetheless, since $\omega$ satisfies $\partial \Bar{\partial} \omega^{n-1} = 0$, the degree is independent of the choice of connection and is therefore well-defined. We should point out that although this definition agrees with the usual topological degree when the metric $g$ is K\"ahler (so that $d\omega = 0$), this degree is not in general a topological degree because it can take values in a continuum. This is indeed the case when $X$ is a non-K\"ahler principal elliptic bundle (see section \ref{Line bundles}).

More generally, we define the {\em degree} of a torsion-free coherent sheaf $\mathcal{E}$ on $X$ by
\[ \deg_g(\mathcal{E}) := \deg_g(\det \mathcal{E}),\]
where $\det \mathcal{E}$ is the determinant line bundle of $\mathcal{E}$, and the {\em slope of $\mathcal{E}$} by 
\[ \mu_g(E) := \deg_g(\mathcal{E})/\rank(\mathcal{E}).\]
We therefore have a notion of slope-stability for sheaves on $X$ with respect to the Gauduchon metric $g$:
\begin{definition}
A torsion-free coherent sheaf $\mathcal{E}$ on $X$ is {\em $g$-(semi)-stable} if and only if $\mu_g(\mathcal{S}) < (\leq) \mu_g(\mathcal{E})$ for every coherent subsheaf $\mathcal{S} \subset \mathcal{E}$ with $0 < \rank(\mathcal{S}) < \rank(\mathcal{E})$.
\end{definition}

\begin{remark}
All line bundles are stable. Moreover, for rank-2 vector bundles on surfaces, it is enough to check stability with respect to line bundles. Finally, stable vector bundles are {\em simple}, that is, $H^0(M,\End(E)) = \mathbb{C}$.
\end{remark}

We also define slope stability for {\em $V$-pairs}, which consist of pairs $(E,\phi)$ made up of a holomorphic vector bundle $E$ on $X$ and a holomorphic section $\phi \in H^0(X,\End E \otimes V)$, called a \emph{Higgs field}, such that $\phi \wedge \phi \in H^0(X,\End E \otimes \wedge^2 V)$ vanishes identically. 
In the special cases where $V=\mc{T}_X$, $V=K_X$, or $V=\mc{T}_X^*$, the $V$-pairs are called \emph{co-Higgs bundles}, \emph{Vafa--Witten pairs}, or \emph{Higgs bundles}, respectively. 
\begin{definition}
Let $(E, \phi)$ be a $V$-pair on $X$. 
A subsheaf $\mc{F}\subseteq E$ is called \emph{$\phi$-invariant} if $\phi(\mc{F})\subseteq \mc{F}\otimes V$. Moreover, the $V$-pair $(E,\phi)$ is said to be \emph{$g$-(semi)-stable} if for any $\phi$-invariant subsheaf $\mc{F}$ of $E$ with $0<\rank{\mc{F}}<\rank{E},$
\[\mu_g(\mc{F})<(\leq) \mu_g(E).\]
We equivalently say that $\phi$ is a \emph{$g$-(semi)-stable Higgs field} for $E$.
\end{definition}
\begin{example}
    If $E$ is a $g$-(semi)-stable bundle on $X$, then any Higgs field $\phi:E \to E\otimes V$ is automatically $g$-(semi)-stable.
\end{example}
Note that if $(E,\phi)$ is a $V$-pair and $\phi=\id_E\otimes \sigma$ for some holomorphic section $\sigma$ of $V$, then $(E,\phi)$ is $g$-(semi)-stable if and only if $E$ is. 
Since $V$-pairs of this type are not interesting from a classification point of view, we typically restrict to $V$-pairs that are trace-free; 
here, we are defining the trace by 
\[\tr\left( \sum\limits_{k=1}^m A_k\otimes \sigma_k \right) :=\sum\limits_{k=1}^m\tr{A_k}\otimes \sigma_k,\] 
where $A_k$ and $\sigma_k$ are local sections of $\End(E)$ and $V$, respectively. Moreover, when $V = T_X$, trace-free Higgs fields give rise to holomorphic Poisson structures on $\mathbb{P}(E)$ when they are {\em strongly integrable} \cite{Polishchuk,Rayan,Matviichuk} (where the strong integrability condition boils down to $\phi \wedge \phi = 0$ when $E$ has rank 2 so that Higgs fields on $E$ correspond to holomorphic Poisson structures on $\mathbb{P}(E)$).
\begin{example}
Let $V$ be any positive-degree $g$-stable holomorphic vector bundle on $X$. 
Although $\mc{O}_X\oplus V$ is not a $g$-stable (or even $g$-semi-stable) vector bundle, 
\[\left(\mc{O}_X\oplus V, \phi:=\begin{bmatrix}0 & 1\\ \alpha & 0\end{bmatrix}:\mc{O}_X\oplus V\to V\oplus (V\otimes V)\right)\]
is a $g$-stable trace-free $V$-pair for any $\alpha \in H^0(X, V\otimes V)$, as any non-trivial $\phi$-invariant subsheaf contains $\mc{O}_X\oplus \mathrm{Im}(\alpha)$ and thus has slope strictly smaller than $\mu_g(\mc{O}_X\oplus V)$.
\end{example}

\subsection{Stable pairs}

In this section, we derive some facts about $V$-pairs $(E,\phi)$ on $X$, where $E$ is a holomorphic vector bundle on $X$ and $\phi \in H^0(X,\End E \otimes V)$. We first consider the case where $V$ is a line bundle.

\begin{proposition}\label{tracefree}
Let $(X,g)$ be a compact Gauduchon manifold. If $V$ is a line bundle on $X$ and $(E,\phi: E \rightarrow E \otimes V)$ is a $g$-stable pair with $\phi \neq 0$, then $\deg V \geq 0$.
\end{proposition}

\begin{proof}
We first note that the Higgs field $\phi: E \rightarrow E \otimes V$ on $E$ induces the Higgs field $\phi' := \phi \otimes \id_V: E \otimes V \rightarrow (E \otimes 
V) \otimes V$ on $E \otimes V$. Moreover, since $V$ is a line bundle, $P$ is a $\phi$-invariant subsheaf of $E$ if and only if $P \otimes V$ is a $\phi'$-invariant subsheaf of $E \otimes V$. Therefore, $(E,\phi)$ is stable if and only if $(E \otimes V,\phi')$ is stable. Also note that $\ker \phi$ is a $\phi$-invariant subsheaf of $E$ and $\Image \phi$ is a $\phi'$-invariant subsheaf of $E \otimes V$. In addition, $\rank(E) = \rank(E \otimes V)$ so that 
\[ \mu(E \otimes V) = \mu(E) + \deg(V).\]
Suppose that $\mu(\Image \phi) = \mu(E \otimes V)$. By stability of $(E \otimes V,\phi')$, we must then have 
\[\rank(\Image \phi) = \rank(E \otimes V) = \rank(E).\]
However, $\Image \phi \simeq E / \ker \phi$ and $\ker \phi$ is a torsion-free subsheaf of $E$ if $\phi$ is not injective, in which case $\rank(\Image \phi) < \rank(E)$. Hence, $\ker \phi = 0$, implying that $\Image \phi \simeq E$. Thus, $\mu(E) = \mu(\Image \phi)$ and $\deg V = 0$ in this case. 

Let us now assume that $\mu(\Image \phi) \neq \mu(E \otimes V)$. Then, $\Image \phi$ is a nonzero torsion-free subsheaf of $E \otimes V$ since $\Image \phi \neq 0$ by assumption. If $\rank(\Image \phi) = \rank(E \otimes V)$, then $\mu(\Image \phi) \leq \mu(E \otimes V)$ so that $\mu(\Image \phi) < \mu(E \otimes V)$. Otherwise, $\Image \phi$ is a $\phi'$-invariant proper subsheaf of $E \otimes V$, which means that $\mu(\Image \phi) < \mu(E \otimes V)$ by stability of $(E \otimes V,\phi')$. Thus, $\mu(\Image \phi) < \mu(E \otimes V)$ in both cases. Consider the exact sequence
\[ 0 \rightarrow \ker \phi \rightarrow E \rightarrow E/\ker \phi \simeq \Image \phi \rightarrow 0.\]
If $\ker \phi = 0$, then $E \simeq \Image \phi$ and $\mu(E) = \mu(\Image \phi)$. Hence,
\[ \mu(E) = \mu(\Image \phi) < \mu(E \otimes V) = \mu(E) + \deg V, \] 
implying that $\deg V > 0$. If instead $\ker \phi \neq 0$, then it is a proper $\phi$-invariant subsheaf of $E$ with $\rank(\ker \phi) < \rank(E)$. Indeed, $\ker \phi \neq E$ and $\Image \phi \neq 0$ since $\phi \neq 0$. And if $\rank(\ker \phi) = \rank E$, then $\Image \phi \simeq E/\ker \phi$ is a non-zero torsion subsheaf of $E$, which is impossible. Thus, $\rank(\ker \phi) < \rank(E)$ and $\mu(\ker \phi) < \mu(E)$ by stability of $(E,\phi)$. Moreover, given the above exact sequence, we obtain $\mu(E)<\mu(\Image \phi)$, implying again that $\mu(E) < \mu(E \otimes V)$ and $\deg V > 0$. Putting all cases together gives $\deg V \geq 0$.
\end{proof}

\begin{remark}
This is a result similar to what is known about (co)-Higgs bundles on curves: if $X$ is a curve, then stable (co)-Higgs bundles with non-zero Higgs field exist on $X$ if and only if $\mc{T}_X^*$ (resp. $\mc{T}_X$) has non-negative degree \cite{Hit-Higgs, Rayan}. Furthermore, the case where $V = K_X$ and $X$ is any compact complex manifold was proven in \cite{MarchesanoMoraruSavelli}.
\end{remark}

In the case where $V$ is the trivial line bundle on $X$, we can say more about $g$-stable pairs:

\begin{proposition}\label{endo}
Let $(X,g)$ be a compact Gauduchon manifold and $(E,\phi)$ be a $g$-stable pair with $\phi \in H^0(X,\End E)$. Then, $\phi = \lambda \id_E$ for some $\lambda \in \mathbb{C}$.
\end{proposition}
\begin{remark}
This proposition gives a slight generalisation of \cite[Theorem 2.1]{Biswas} and clarifies the proof of \cite[Lemma 3.2]{MarchesanoMoraruSavelli}.
\end{remark}
\begin{proof}
Suppose that $\phi \neq 0$ so that $\Image \phi \neq 0$. If $\ker \phi = 0$, then $E \simeq \Image \phi$ and $\Image \phi$ is a subbundle of $E$ with $\rank(\Image \phi) = \rank(E)$, implying that $\phi$ is an automorphism. Let us assume instead that $\ker \phi \neq 0$ so that $\ker \phi$ is a torsion-free subsheaf of $E$ and $\rank(\Image \phi) = \rank(E/\ker \phi) < \rank(E)$. Moreover, since $\Image \phi \neq 0$, it is a torsion-free subsheaf of $E$, implying that $\rank(\ker \phi) < \rank(E)$. Hence, $\ker \phi$ is a $\phi$-invariant proper subsheaf of $E$ that fits into the exact sequence  
\[ 0 \rightarrow \ker \phi \rightarrow E \rightarrow E/\ker \phi \simeq \Image \phi \rightarrow 0.\]
By stability of $(E,\phi)$, we have $\mu(\ker \phi) < \mu(E)$, and so $\mu(E)<\mu(\Image \phi)$. On the other hand, $\Image \phi$ is a non-zero $\phi$-invariant subsheaf of $E$. Moreover, if $\rank(\Image \phi) = \rank(E)$, then $\mu(\Image \phi) \leq \mu(E)$; and if $\rank(\Image \phi) < \rank(E)$, then $\Image \phi$ is a proper subsheaf of $E$, implying that $\mu(\Image \phi) < \mu(E)$ by stability of $(E,\phi)$. Hence, $\mu(E) < \mu(\Image \phi) \leq \mu(E)$, leading to a contradiction. Thus, $\ker \phi = 0$ and $\phi$ is a vector bundle automorphism of $E$.

Let $p\in X$ be any point. Then $\phi\vert_{E_p}$ is a finite-dimensional operator over $\mathbb{C}$, so it has an eigenvalue $\lambda \in \mathbb{C}^*$. Note that $(E,\phi)$ is stable if and only if $(E, \phi'=\phi-\lambda\id_E)$ is (because any subsheaf $P$ of $E$ is $\phi$-invariant if and only if it is $\phi'$-invariant), and $p$ is contained in the support of $\mathrm{coker}(\phi')$. Since $\mathrm{coker}(\phi')$ has non-empty support and $\phi'$ is a map from $E$ to itself, $\chern(\ker(\phi'))=\chern(\mathrm{coker}(\phi'))\neq 0$, meaning that $\ker(\phi')\neq 0$. As we have already eliminated the case where $\ker(\phi')\neq 0$ and $\phi'\neq 0$, we must have $\phi=\lambda\id_E$.
\end{proof}

Finally, if $V$ is an extension of vector bundles, we have the following:
\begin{proposition}\label{extension}
Let $(X,g)$ be a compact Gauduchon manifold. Suppose that $V$ is an extension of holomorphic vector bundles $$0 \rightarrow V_1 \stackrel{i}{\rightarrow} V \stackrel{p}{\rightarrow} V_2 \rightarrow 0$$
on $X$. If $X$ has no non-trivial $g$-stable pairs of the form $(E, \vp_1:E \to E\otimes V_1)$ or $(E, \vp_2:E \to E\otimes V_2)$, then it has no non-trivial $g$-stable pairs of the form $(E,\phi: E \to E\otimes V)$.

\end{proposition}
\begin{proof}
Suppose that every trace-free $g$-stable pair $(E, \vp_i:E \to E\otimes V_i)$ has $\vp_i=0$, where $i=1,2$. Let $\phi:E \to E\otimes V$ be a trace-free stable pair, and set $\vp_2:=(\id_E\otimes p)\circ \phi$. Suppose $P_2$ is a $\vp_2$-invariant proper subsheaf of $E$. Then $\vp_2(P_2)\subset P_2\otimes V_2$, so $\phi(P_2)\subset P_2\otimes V$, meaning that $P_2$ is also $\phi$-invariant. By stability of $(E,\phi)$, we have $\mu(P_2)<\mu(E)$ so that $(E,\vp_2)$ is also stable. Since $\vp_2$ is clearly trace-free, $\vp_2=0$, implying that $\phi=(\id_E\otimes i)\circ\vp_1$ for some trace-free $\vp_1:E \to E\otimes V_1$. Suppose $P_1$ is a $\vp_1$-invariant proper subsheaf of $E$. Then $\vp_1(P_1)\subset P_1\otimes V_1\subset P_1\otimes V$, so $\mu(P_1)<\mu(E)$ by stability of $(E,\phi)$. This implies that $(E,\vp_1)$ is stable, so $\phi=(\id_E\otimes i)\circ\vp_1=0$.
\end{proof}

\section{Bundles on principal elliptic surfaces}\label{elliptic}
In this section, we present for the convenience of the reader some facts about holomorphic vector bundles on principal elliptic fibrations that can be found in \cite{Teleman,MoraruThesis,BrMor,BrMor2,BrMoFM,Boulter}. We also prove, in section \ref{Vpairs}, some general results for $V$-pairs on non-K\"ahler principal elliptic surfaces.

Recall that every principal elliptic surface can be described in the following way \cite{Teleman}:
Let $B$ be a compact Riemann surface and $\Theta\in \Pic^d(B)$. We denote the total space of the holomorphic frame bundle corresponding to $\Theta$ by $\Theta^*$. (As a complex manifold, $\Theta^*$ is the complement of the zero section in the total space of $\Theta$.) If $\tau$ is a complex number with $\lvert\tau\rvert >1$, then $X:=\Theta^*/(\tau)$ is a principal elliptic surface with projection onto $B$ induced from the projection of $\Theta$, where $(\tau)$ acts as a subgroup of the structure group $\mathbb{C}^*$ of $\Theta^*$. We denote by $T$ the structure group $\mathbb{C}^*/(\tau)$ of the elliptic surface.

The above description uniquely determines the biholomorphism class of the resulting elliptic surface up to swapping $\Theta$ and $\Theta^\vee$ and pulling back by an isomorphism of the base. Furthermore, $X$ is K\"ahler if and only if $d:=c_1(\Theta)=0$. (Up to homeomorphism, the elliptic surface is uniquely determined by $\lvert d\rvert$ and the genus of $B$.)

\subsection{Line bundles}\label{Line bundles}
Let $\pi:X \to B$ be a principal elliptic surface. 
In order to understand the Picard group of $X$, we break up the problem into studying $\Pic^0(X)$ and the Néron--Severi group $NS(X):=\Pic(X)/\Pic^0(X)$. 

Any line bundle in $\Pic^0(X)$ decomposes uniquely as the product of the pullback of a degree zero line bundle on the base and a bundle with \emph{constant factor of automorphy} \cite[Proposition 1.6]{Teleman}. 
The line bundle $L_a$ with constant factor of automorphy $a$ is the quotient of the trivial line bundle $\Theta^*\times \mathbb{C}$ on $\Theta^*$ by the $\mathbb{Z}$-action \begin{align*}\Theta^*\times \mathbb{C}\times \mathbb{Z}&\to \Theta^*\times \mathbb{C},\\
(z,t,n)&\mapsto(\tau^nz,a^nt).\end{align*}
Using these facts, one can also show that $\pi^*\Theta\simeq L_{\tau^{-1}}$ \cite[Proposition 1.7]{Teleman}, demonstrating that the first Chern class of a pullback bundle is torsion. We can see from the above computations that $\Pic^\mathrm{Tors}(X)\simeq \Pic(B)\times \mathbb{C}^*/(\pi^*\Theta, L_\tau)$, where $\Pic^\mathrm{Tors}(X)$ is the subgroup of $\Pic(X)$ consisting of line bundles with torsion first Chern class.

Given that $\pi:X \to B$ is an elliptic fibration, we can understand the quotient $\Pic(X)/\pi^*\Pic(B)$ as a group corresponding to families of line bundles on the elliptic fibre $T$ parameterised by $B$. 
Since the line bundles with constant factor of automorphy contain a cyclic subgroup of $\pi^*\Pic(B)$, such families can be described by maps from $B$ to $T^*:=\Pic^0(T)$. 
Up to fixing base points, we obtain the group $\Pic(X)/\pi^*\Pic(B)\simeq T^*\times \Hom(J(B), T^*)$, where $J(B)$ is the Jacobian variety of $B$. 
(For a more detailed proof of this result, see \cite[Section 3.2]{Brin}.)

There is a one-to-one correspondence between $T^*\times \Hom(J(B), T^*)$ and maps from $B$ to $T^*$, as the Albanese map of $B$ is $b\mapsto \mc{O}_B(b-b_0)$ for a choice of base point $b_0 \in B$. 
This fact allows us to parameterise $\Pic(X)/\pi^*\Pic(B)$ by the set of sections of the \emph{relative Jacobian} $\pi_J:J(X)\to B$ when $\pi:X\to B$ is principal. (Note that the relative Jacobian of a non-Kähler elliptic surface $\pi:X\to B$ with general fibre $T$ is $J(X)=B\times T^*$, with projection onto the first factor being the associated map to $B$.) 
To any line bundle $\delta \in \Pic(X)$, we associate the section \[S_\delta=\{(b,\lambda)\in J(X): \delta\vert_{\pi^{-1}(b)}\simeq \lambda\},\]
called the \emph{spectral curve} of $\delta$.
Let $\pi: X \rightarrow B$ be a non-K\"ahler principal bundle over a curve $B$ of genus $g$. The degree of torsion line bundles can be computed explicitly in this case (for details, see \cite{Teleman}). Indeed, every torsion line bundle $L \in \Pic^\tau(X)$ decomposes uniquely as $L = H \otimes L_a$ for $H \in \cup_{i=0}^{d-1}\Pic(B)$ and $a \in \mathbb{C}^*$, where $d$ is the degree of the positive line bundle on $B$ used to construct $X$. The degree of $L$ is given, in terms of this decomposition, by
\[ \deg_g L = c_1(H) - \frac{d}{\ln \vert \tau \vert} \ln \vert a \vert.\]
In particular, $\deg(\pi^* H) = \deg H$ for all $H \in \Pic(B)$. Moreover, note that $\deg_g: \Pic^\tau(X) \rightarrow \mathbb{R}$ is surjective since, for any $c \in \mathbb{R}$, we have $c = \deg_g L_a$ with $a = \tau e^{-c/d}$.

Associated to any line bundle $\delta \in \Pic(X)$, there is also a ruled surface $\mathbb{F}_\delta$ given by the quotient of $J(X)$ by the $\mathbb{Z}/2\mathbb{Z}$-action sending $(b,\lambda)$ to $(b,\delta\vert_{\pi^{-1}(b)}\otimes \lambda^{-1})$.
The $e$-invariant 
\begin{equation}\label{ruled}e_\delta:=\max\{-\sigma^2:\sigma \text{ is a section of }\mb{F}_\delta\}\end{equation}
of $\mb{F}_\delta,$ where $\sigma^2$ is the self-intersection of the curve $\sigma$, appears in the existence criteria for rank-$2$ vector bundles on $X$ with determinant $\delta$, as shown in \cite[Theorem 4.5]{BrMor}. In particular, there is a rank-2 bundle $E$ on $X$ with determinant $\delta$ and second Chern class $c_2$ if and only if $\Delta(2,\delta,c_2)\geq -e_\delta/4$. Note that the e-invariant lies in the range $-g\leq e_\delta \leq 0$. 

\subsection{The holomorphic (co)tangent bundle}

In this section, we describe the holomorphic tangent and cotangent bundles of principal elliptic surfaces.
\begin{proposition}\label{cotangent}
Let $\pi:X\to B$ be a principal elliptic surface. The holomorphic cotangent bundle $\mathcal{T}_X^*$ is then given by $\pi^*(\mathcal{A}_B)$, where $\mathcal{A}_B$ is an extension $$0 \rightarrow K_B \rightarrow \mc{A}_B \rightarrow \mc{O}_B \rightarrow 0.$$ If $X$ is K\"ahler, the extension splits; otherwise, $\mathcal{A}_B$ is the unique non-split extension of $\mathcal{O}_B$ by $K_B$.
\end{proposition}
\begin{proof}
Let $\{U_\alpha, \varphi_\alpha\}$ be an atlas of $B$ such that $\{U_\alpha, \psi_\alpha\}$ is a local trivialization for $\Theta^*$ and set $g_{\alpha\beta}:\varphi_\beta(U_\alpha\cap U_\beta)\to \varphi_\alpha(U_\alpha\cap U_\beta)$ to be the corresponding transition functions. We also set $V_0=\{[w]\in \mathbb{C}^*/(\tau):1<\lvert w \rvert <\lvert\tau\rvert\}$ and $V_1=\{[w]\in \mathbb{C}^*/(\tau):\lvert\tau\rvert^{-1/2}<\lvert w\rvert<\lvert\tau\rvert^{1/2}\}$. Let $f_{\alpha\beta}$ be transition maps for $\Theta$. Then $\{\psi_\alpha^{-1}(\vp_\alpha(U_\alpha)\times V_i)\}$ gives an atlas for $X$ with transition functions $\eta_{ij,\alpha\beta}(z,w)=(g_{\alpha\beta}(z), f_{\alpha\beta}(z)\tau^k w)$, where $z$ is a local coordinate in $B$, $w$ is a coordinate of $\mathbb{C}^*$, and $$k=\begin{cases}
1 & i=0,j=1,\lvert w\rvert<1\\
-1 & i=1,j=0, \lvert w\rvert>\lvert \tau\rvert^{1/2}\\
0 & \text{otherwise}
\end{cases}.$$ 
Note that $\diff z$ and $\frac{1}{w}\diff w$ give a local frame for the open set $\psi_\alpha^{-1}(\varphi_\alpha(U_\alpha)\times V_i).$ If we now compute the pullbacks $\eta_{ij,\alpha\beta}^*(\diff \tilde{z})$ and $\eta_{ij,\alpha\beta}^*(\frac{1}{\tilde{w}}\diff \tilde{w})$, we get
\begin{align*}
    \eta_{ij,\alpha\beta}^*(\diff \tilde{z})&=\diff(g_{\alpha\beta}(z))\\
    &=g_{\alpha\beta}' \diff z,\\
    \eta_{ij,\alpha\beta}^*(\frac{1}{\tilde{w}}\diff \tilde{w})&=\frac{1}{f_{\alpha\beta}\tau^k w}\diff(f_{\alpha\beta}\tau^k w)\\
    &=\frac{1}{w}\diff w+ \frac{f_{\alpha\beta}'}{f_{\alpha\beta}}\diff z,
\end{align*}
so $T^*_X$ has transition maps $\begin{bmatrix}g_{\alpha\beta}' & \frac{f_{\alpha\beta}'}{f_{\alpha\beta}}\\ 0 & 1\end{bmatrix}.$ Since the transition maps of $\mathcal{T}^*_X$ depend only on $z$ and since $g_{\alpha\beta}'$ are the transition maps of $\pi^*(\mathcal{T}^*_B)=K_X$, we find that $\mathcal{T}^*_X$ is the pullback of an extension in $\Ext^1(\mc{O}_B, K_B)$. If $c_1(\Theta)=0$, then without loss of generality $f_{\alpha\beta}$ may be chosen to be locally constant, so the extension splits. If instead $c_1(\Theta)\neq 0$, it is a non-split extension. By Serre duality, $$\Ext^1(\mc{O}_B, K_B)=\Hom(K_B, \mc{O}_B\otimes K_B)^*=\Hom(K_B,K_B)^*\cong \mathbb{C},$$
so this extension is unique.
\end{proof}
\begin{remark}
In the case where $\pi:X\to \mb{P}^1$ is a Hopf surface, $\mc{A}_B\cong \mc{O}_{\mb{P}^1}(-1)\oplus \mc{O}_{\mb{P}^1}(-1)$. If $\pi:X\to B$ is a Kodaira surface, then $\mc{A}_B$ is the Atiyah bundle on $B$.
\end{remark}

The above proposition also determines the holomorphic tangent bundle $\mathcal{T}_X$, as $\mathcal{T}_X\cong \mathcal{T}_X^*\otimes K_X^\vee=\mathcal{T}_X^*\otimes\pi^*(K_B^\vee)$.

If we use Propositions \ref{tracefree}, \ref{endo}, and \ref{extension} with this description of the (co)-tangent bundle, we get the following:


\begin{proposition}
Let $\pi:X\to B$ be a principal elliptic surface. If $X$ admits a non-zero stable trace-free (co)-Higgs bundle, then so does $B$. In other words, if $X$ has a non-zero stable trace-free Higgs bundle, then $B$ has genus at least 2, and if $X$ has a non-zero stable trace-free co-Higgs bundle, then $B$ has genus 0. 
\end{proposition}

\begin{proof}
Suppose first that $X$ has a non-trivial stable co-Higgs bundle. Then $\mc{T}_X\cong (\pi^*\mc{A}_B)^\vee$ by Proposition \ref{cotangent}, so it is an extension of the anti-canonical bundle by $\mc{O}_X$. As shown in Proposition \ref{endo}, there are no non-zero stable pairs $(E,\phi)$ with $\phi \in \End_0(E)$, so we conclude that $X$ admits a non-zero trace-free stable pair $(E,\phi:E \to E\otimes K_X^{-1})$ by Proposition \ref{extension}. This is only possible if $B$ has genus $0$.
If we now assume that that $X$ has a non-trivial stable Higgs bundle, a similar argument shows that $B$ has genus at least $2$.
\end{proof}

\begin{remark}
If $\pi:X\to B$ is a principal elliptic surface, any Gauduchon metric $g$ on $X$ can be normalised so that $\deg_g(\pi^*(H))=\deg(H)$ for all $H \in \Pic(B)$ \cite{LuTe}.
\end{remark}

\subsection{Bundles on non-K\"ahler elliptic surfaces}

We now restrict ourselves to the non-K\"ahler case. Let $\pi: X \rightarrow B$ be a non-K\"ahler principal elliptic surface, with $B$ a smooth compact connected curve. Let $T$ be the fibre of $\pi$, which is an elliptic curve, and denote its dual $T^*$ (we fix a non-canonical identification $T^* := \Pic^0(T)$). In this case, the relative Jacobian of $\pi: X \rightarrow B$ is simply $J(X) = B \times T^* \stackrel{p_1}{\rightarrow} B$.

\subsubsection{Filtrability and destabilising bundles}
Rank-2 holomorphic vector bundles on compact non-algebraic complex surfaces break up into two categories: those which are filtrable, and those which are non-filtrable, where filtrability is defined as follows:
\begin{definition}
  A rank-2 sheaf $\mc{E}$ on a complex surface is \emph{filtrable} if it admits a rank-1 subsheaf.
  A rank-2 sheaf is \emph{non-filtrable} if every non-zero subsheaf has rank two.
\end{definition}
\begin{remark}
This distinction is unique to non-algebraic complex surfaces;
a rank-two coherent sheaf defined over an algebraic surface is always filtrable.
\end{remark}

To study the existence of filtrable bundles, we introduce the invariant \[m(2,\delta):=-\frac{1}{2} \max\left\{\left(\frac{c_1(\delta)}{2}-\mu\right)^2: \mu\in NS(X)\right\}.\]
(The invariant $m(r,\delta)$ generally depends on the rank $r$, but has a simpler formula when $r=2$.) A result of B\u{a}nic\u{a} and Le Potier \cite[Th\'{e}or\`{e}me 2.3]{BanicaLePotier} states that there is a rank-2 filtrable bundle on a compact non-K\"ahler surface $X$ with determinant $\delta$ and second Chern class $c_2$ if and only if $\Delta(2,\delta,c_2)\geq m(2,\delta)$. Note that if $X$ is a non-K\"ahler principal elliptic surface with genus 1 fibre $T$ and  base $B$, and $\delta$ is chosen so that $c_1(\delta)$ corresponds to a finite map $\gamma:B\to T$ of minimal degree, then \[m(2,\delta)=-\frac{c_1(\delta)^2}{8}=\frac{\deg(\gamma)}{4}.\]
\begin{remark}\label{m-invariant}
For $B$ of genus 1, the curves $B$ and $T$ can be chosen to have $\deg(\gamma)$ be any positive integer by standard elliptic curve theory. And for $B$ of genus 2, the curves $B$ and $T$ can be chosen to have the degree of the minimal map $\gamma$ be any integer strictly greater than $1$ \cite[Theorem 3]{Kani}. Since $m(2,\delta)=\deg(\gamma)/4$ in such cases, this means that $\delta$ can be chosen to make $m(2,\delta)$ arbitrarily large.
\end{remark}

Let $\mc{E}$ be a rank-2 sheaf with determinant $\delta$ and second Chern class $c_2$. We say that $\mc{E}$ (or alternatively its determinant and second Chern class) is \emph{in the non-filtrable range} if \[-\frac{e_\delta}{4}\leq \Delta(2,\delta,c_2)<m(2,\delta).\] The first of the two inequalities ensures that bundles with these invariants exist, and the second inequality forces any such bundle to be non-filtrable.
\begin{remark}
    In the genus-2 case, we can always choose $B$, $T$, and $\delta$ so that the non-filtrable range is non-empty by Remark \ref{m-invariant}.
\end{remark}

Clearly, if a rank-2 bundle is non-filtrable, then it is automatically stable with respect to any Gauduchon metric.
Stability of reducible bundles can be investigated using their maximal destabilising bundles.
\begin{definition}
    Let $X$ be a compact complex surface with Gauduchon metric $g$, and let $E$ be a rank-2 vector bundle on $X$. A locally-free subsheaf $F\subseteq E$ is a \emph{maximal destabilising bundle} for $E$ with respect to $g$ if $\rank{F}=1$, $E/F$ is torsion-free, and for any positive-degree line bundle $L$, we have $H^0(X,\Hom(F\otimes L, E))=0$.
\end{definition}

Maximal destabilising bundles satisfy the following properties:
\begin{proposition}[Friedman \cite{Fried}]
Let $E$ be a rank-2 vector bundle on $X$ and $\mc{F}\subseteq E$ be a subsheaf with $0<\rank{\mc{F}}<\rank{E}$. Then: 
\begin{enumerate}
\item
There exists an effective divisor $D$ such that $\mc{F}^{\vee\vee}\otimes \mc{O}_X(D)$ is a maximal destabilising bundle of $E$.
\item
For any maximal destabilising bundle $L$ of $E$, there is a finite analytic subspace $Z\subset X$ such that 
\[0 \rightarrow L \rightarrow E \rightarrow \det(E)\otimes L^{-1}\otimes \mc{I}_Z \rightarrow 0\]
is an exact sequence, where $\mc{I}_Z$ is the ideal sheaf of $Z$.
\end{enumerate}
\end{proposition}

For certain classes of surfaces, one can compute the destabilising bundles, which simplifies the process for checking stability.
In the case of non-K\"ahler elliptic surfaces, we have the following:
\begin{proposition}[Brînz\u{a}nescu--Moraru \cite{BrMor2}, Proposition 3.3 and Theorem 3.5]
    If $\pi:X\rightarrow B$ is a non-Kähler elliptic surface and $E$ is a rank-2 vector bundle on $X$, then $E$ has at most two maximal destabilising bundles, and a maximal destabilising bundle $K$ for $E$ is unique if and only if $\det(E)\otimes K^{-2}\in \pi^*\Pic(B)$. If $E$ is an extension of line bundles with two distinct destabilising bundles $K_1, K_2$, they satisfy the relation \[K_1\otimes K_2 \otimes \det(E)^{-1}\simeq \mc{O}_X\left(\sum\limits_{F \in A(E)} -F\right),\] where $A(E)$ is the set of fibres $F$ of $\pi$ such that $E\vert_F$ is not split. When $E$ is not an extension of line bundles, there is an allowable elementary modification of $E$ with the same maximal destabilising bundles.
\end{proposition}
\noindent
 Elementary modifications will be discussed in detail in Section \ref{elementary}.

\subsubsection{The spectral construction}\label{Spectral}
To study bundles on $X$, one of our main tools is restriction to the fibres of $\pi: X \rightarrow B$. It is important to point out that since $X$ is non-K\"ahler, the restriction of {\em any} vector bundle on $X$ to a fibre of $\pi$ {\em always} has trivial first Chern class. A holomorphic vector bundle $E$ on $X$ is then semistable on the generic fibre of $\pi$. In fact, the restriction of $E$ to a fibre $\pi^{-1}(b)$ is unstable on at most an isolated set of points $b \in B$; these isolated points are called the {\em jumps} of the bundle. Furthermore, there exists a divisor $S_E$ in the relative Jacobian $J(X)$ of $X$, called the {\em spectral curve} or {\em cover} of the vector bundle $E$, that encodes the isomorphism class of $E$ over the fibres of $\pi$. 

The spectral curve can be described more concretely as 
\[S_E=\{(b,\lambda) \in B\times T^*: H^1(T, E\vert_{\pi^{-1}(b)}\otimes \lambda)\neq 0\},\] 
with the multiplicity of $(b,\lambda)$ given by $h^1(T, E\vert_{\pi^{-1}(b)}\otimes \lambda)$. 
Since the restriction morphism 
\[\iota^*:H^2(X,\mathbb{Z})\to H^2(T,\mathbb{Z})\]
is zero \cite{Teleman}, we have that $E\vert_{\pi^{-1}(b)}\otimes \lambda$ is a degree zero bundle for all $(b,\lambda) \in B\times T^*$. 
This fact together with the classification of vector bundles on genus 1 curves gives that $S_E\cap \{b\}\times \Pic^0(T)$ contains at most $\rank(E)$ points if and only if $E\vert_{\pi^{-1}(b)}$ is semistable. 
The spectral curve thus has the form 
\[S_E=\overline{C}+\sum\limits_{b \in Z_E}\ell_b(\{b\}\times T^*)\]
for some positive integers $\ell_b$, where $\overline{C}$ is an $r$-section of $J(X)\to B$ and 
\[Z_E:=\{b \in B: E\vert_{\pi^{-1}(b)} \text{ is unstable}\}.\]

\begin{remark}
    The spectral curve of a holomorphic vector bundle $E$ on $X$ can also be constructed using a twisted Fourier--Mukai transform between the category of coherent sheaves on $X$ and the category of twisted sheaves on $J(X)$ as described in \cite{BrMoFM}.
\end{remark}

\subsubsection{Spectral curves and elementary modifications}\label{elementary}
In order to understand vector bundles with jumps, the main method is to study their \emph{elementary modifications}. 
Given a rank-2 vector bundle $E$ on a complex manifold $X$, a smooth effective divisor $D$, a line bundle $\lambda$ on $D$, and a surjective sheaf map $g:E\vert_D\to \lambda$, the \emph{elementary modification $E'$ of $E$ by $(D,\lambda)$} is the unique vector bundle satisfying the exact sequence 
\[ 0 \rightarrow  E' \rightarrow E \rightarrow \iota_*\lambda \rightarrow 0, \]
where $\iota:D\to X$ is the inclusion map. 
The invariants of an elementary modification are given by 
\begin{align*}\det(E')=\det(E)\otimes \mc{O}_X(-D), && c_2(E')=c_2(E)+c_1(E).[D]+\iota_*c_1(\lambda).\end{align*}
In the case that $\pi:X\to B$ is a non-K\"ahler elliptic surface and $D$ is a prime divisor, the divisor $D$ is of the form $\pi^{-1}(b)$ for some $b \in B$. 
Since $D$ has torsion first Chern class, the determinant and second Chern class are related by 
\begin{align*}
	\det(E')=\det(E)\otimes \pi^*(\mc{O}_B(-b)), && c_2(E')=c_2(E)+\deg(\lambda),
\end{align*}
and the elementary modification $E'$ has $\ell(E',b)=\ell(E,b)+\deg(\lambda)$ in this case.

If a vector bundle $E$ has a jump at $b$, there is a unique elementary modification of $E$ along $\pi^{-1}(b)$ by a negative degree bundle, called the \emph{allowable elementary modification} of $E$ at $b$ \cite[Section 4.1.2]{MorHopf}; 
in particular, since $E\vert_{\pi^{-1}(b)}$ is unstable, it is of the form $\lambda \oplus (\lambda^*\otimes \det(E)\vert_{\pi^{-1}(b)})$ for some $\lambda \in \Pic^{-h}(T^*)$ with $h>0$, with the map $g:E\vert_{\pi^{-1}(b)}\to \lambda$ given by projection onto the first coordinate. The positive integer $h$ is called the {\em height} of the jump. Moreover, we can associate to $E$ a finite sequence $\{ \bar{E}_1,\bar{E}_2, \dots, \bar{E}_l \}$ of allowable elementary modifications such that $\bar{E}_l$ is the only element in the sequence that does not have a jump at $T$. The integer $l$ is called the {\em length} of the jump at $T$. Note that if $h_0= h$ is the height of $E$ and $h_i > 0$ is the height of the elementary modification $\bar{E}_i$ for $1 \leq i \leq l-1$, then 
\[ h_0 \geq h_1 \geq h_2 \geq \cdots \geq h_{l-1}\]
and 
\[ \mu = \sum_{i=0}^{l-1} h_i.\]
The integer $\mu$ is called the {\em multiplicity} of the jump. Note that the heights $h_0, \cdots , h_{l-1}$ give a partition of the multiplicity $\mu$ into $l$ parts. We denote by $s$ be the number of distinct sizes of these parts. For example, if $h_0 = \dots = h_{l-1}$, then $s=1$. Whereas, if $l=4$ and $h_0 = 4$, $h_1 = h_2 = 2$, $h_3 = 1$, then $s=3$.

The above discussion tells us that any jump can be removed by performing a finite number of allowable elementary modifications. Consequently, we have:

\begin{proposition}[Br\^inz\u{a}nescu--Moraru \cite{BrMor}]\label{FiniteUnstable}
	If $E$ is a rank-$2$ vector bundle on $X$, then $E$ has finitely many jumps, and 
 \[\sum\limits_{b\in Z_E} \ell(E, b)\leq 2\Delta(E).\]
\end{proposition}

By contrast, elementary modifications by positive-degree line bundles are highly non-unique; in fact, if the restriction of a vector bundle $E$ to a fibre $\pi^{-1}(b)$, $b \in B$, is such that 
\[ E\vert_{\pi^{-1}(b)}\cong L\otimes (L^*\otimes \det(E)\vert_{\pi^{-1}(b)})\] with $L \in \Pic^h(T^*)$, $h\geq 0$, $L\not\cong L^*\otimes \det(E)\vert_{\pi^{-1}(b)}$, 
then $E$ has an elementary modification at $b$ by $\lambda$ for every $\lambda \in \Pic^r(T^*)$ with $r\geq h$ \cite[Section 4.1.3]{MorHopf}. 
Finally, for the case of an elementary modification by a degree zero line bundle, the behaviour of the elementary modification depends on whether the initial bundle is regular.
\begin{definition}\label{regular-def}
	A rank-2 vector bundle $E$ on $X$ is \emph{regular at $b$} for some $b \in B$ if $E\vert_{\pi^{-1}(b)}$ is semi-stable and not isomorphic to $\lambda \oplus \lambda$ for any $\lambda \in \Pic(T)$. Moreover, $E$ is \emph{regular} if it is regular at $b$ for every $b \in B$.
\end{definition}

We finish with the following facts about non-filtrable bundles:
\begin{proposition}\label{smooth-spectral-curve}
	Let $E$ be a rank-2 vector bundle. If $E$ is in the non-filtrable range, then $E$ is regular away from jumps. Moreover, if the spectral curve $S_E$ contains no jumps, then $S_E$ is smooth and the ramification divisor $R$ of $\rho:S_E\to B$ satisfies $\deg{R}=8\Delta(E)$.
\end{proposition}
\begin{proof}
    See \cite[Propositions 3.16 and 3.17]{Boulter}
\end{proof}

\subsubsection{Higgs fields}
\label{Vpairs}


In this section, we consider Higgs fields $\phi: E \rightarrow E \otimes V$ where $V = \pi^* W$ for some vector bundle $W$ on $B$.

\begin{proposition}\label{fitrable}
Let $E$ be a rank-2 filtrable vector bundle on $X$ with maximal destabilising bundles $K_1$ and $K_2$. Set $H:=\pi_*(\det(E)^{-1}\otimes K_1\otimes K_2)$.
\begin{enumerate}
    \item[(a)] Suppose that $E$ is regular on the generic fibre of $\pi$. Then,
    \[ h^0(X,\End_0(E)\otimes V) = h^0(B, H \otimes W).\]
    In particular, $H^0(X,\End_0 E\otimes V) \simeq H^0(X,\Hom(E,K_i \otimes V))$ for $i=1,2$, implying that $K_1$ and $K_2$ are both $\phi$-invariant for all $\phi \in H^0(X,\End E\otimes V)$. A Higgs field $\phi: E \rightarrow E \otimes V$ is thus stable if and only if $E$ is.
    
    \item[(b)] Suppose that $E$ is not regular on the generic fibre of $\pi$ so that $K = K_1 = K_2$ is its unique maximal destabilising bundle. We have two cases:
    \begin{enumerate}
        \item[(i)] If $E$ is an extension of line bundles, then $E\simeq K \otimes \pi^*(F)$ for some rank-2 vector bundle $F$ on $B$ that is an extension of $H^{-1}$ by $\mathcal{O}_B$. Furthermore, $ H^0(X,\End E\otimes V)\simeq H^0(B, \End F\otimes W)$ so the Higgs fields on $E$ twisted by $V$ are precisely the pullbacks of Higgs fields on $F$ twisted by $W$.
        
        \item[(ii)] If $E$ is not an extension of line bundles, then 
        \[ h^0(X,\End_0(E)\otimes V)\geq h^0(X,\Hom(E,K \otimes V)) \geq h^0(B,H\otimes W).\] 
        In fact, $H^0(X,\End_0 E\otimes V) \simeq H^0(X,\Hom(E,K \otimes V))$ if $h^0(B,\pi_*(K^{-1} \otimes E) \otimes W) = h^0(B,W)$, in which case $K$ is $\phi$-invariant for all Higgs fields $\phi: E \rightarrow E \otimes V$, and Higgs fields are stable if and only if $E$ is. 
    \end{enumerate}
\end{enumerate}
\end{proposition}

\begin{remark}
By \cite[Proposition 3.4]{BrMor2}, the bundle $\det(E)^{-1}\otimes K_1\otimes K_2=\pi^*(H)$ for some $H \in \Pic(B)$ whenever $E$ is filtrable. Moreover, $h^0(B,\pi_*(K^{-1} \otimes E) \otimes W) = h^0(B,W)$ when $h^0(X,\pi^* H^{-1} \otimes V) = 0$.
\end{remark}

\begin{proof}
Since $E$ is filtrable with maximal destabilising bundle $K_1$, it fits into an exact sequence of the form
\begin{equation}\label{exact-filtrable-Higgs}
 0 \rightarrow K_1 \stackrel{i}{\rightarrow} E \stackrel{p}{\rightarrow} \delta\otimes K_1^{-1}\otimes \mc{I}_Z \rightarrow 0,   
\end{equation}
where $\delta=\det(E)$ and $Z$ is a zero-dimensional subset of $X$. 

Let us first assume that $E$ is regular on the generic fibre of $\pi$. Then, $\pi_*(K_1^{-1}\otimes E)=\pi_*(K_2^{-1}\otimes E)=\mc{O}_B$.
To determine the dimension of the space of trace-free Higgs fields on $E,$ we tensor the exact sequence \eqref{exact-filtrable-Higgs} by $E^\vee\otimes V=\delta^{-1}\otimes E\otimes V$ and look at cohomology. This gives the left-exact sequence 
\[ 0 \rightarrow H^0(X,\\Hom(E, K_1 \otimes V)) \rightarrow H^0(X,\End(E)\otimes V) \rightarrow H^0(X,K_1^{-1}\otimes E\otimes V \otimes \mc{I}_Z).\]
Note that for any element of $H^0(X,\End(E)\otimes V)$ of the form $\id_E\otimes s$ with $s \in H^0(X,V)$, its image in $H^0(X, K_1^{-1}\otimes E \otimes V \otimes \mc{I}_Z)$ is zero if and only if $s=0$, so 
\[ h^0(X,K_1^{-1}\otimes E\otimes V \otimes \mc{I}_Z)\geq h^0(X,V)  = h^0(B,W).\] 
Since $K_1^{-1}\otimes E\otimes V \otimes \mc{I}_Z$ is a subsheaf of $K_1^{-1}\otimes E\otimes V$, we also have 
\begin{align*}
h^0(X,K_1^{-1}\otimes E\otimes V \otimes \mc{I}_Z) &\leq h^0(X, K_1^{-1}\otimes E\otimes V)\\
&=h^0(B, \pi_*(K_1^{-1}\otimes E)\otimes W) = h^0(B, W).
\end{align*}
Therefore, $H^0(X,V) \simeq  H^0(X, K_1^{-1}\otimes E\otimes V\otimes \mc{I}_Z)$ and $H^0(X,\End_0(E)\otimes V) \simeq H^0(X, \Hom(E, K_1 \otimes V))$. Note that
\begin{align*} 
\Hom(E, K_1 \otimes V) &\simeq E^\vee \otimes K_1 \otimes V \simeq E \otimes \delta^{-1} \otimes K_1 \otimes V \simeq E \otimes K_2^{-1}\otimes \pi^*(H)\otimes V \\
&\simeq (K_2^{-1} \otimes E) \otimes \pi^*(H \otimes W).
\end{align*}
Consequently,
\begin{align*}
H^0(X,\End_0(E)\otimes V)&\simeq H^0(X,(K_2^{-1} \otimes E) \otimes \pi^*(H \otimes W))\\
&= H^0(B, \pi_*(K_2^{-1}\otimes E)\otimes H\otimes W) = H^0(B, H\otimes W).
 \end{align*}
Since every non-trivial Higgs field $\phi$ on $E$ is of the form 
\[ \phi=(\iota\otimes \id_V)\circ \psi-\frac{1}{2}\mathrm{tr}((\iota\otimes \id_V)\circ \psi)\] 
for some $\psi \in \Hom(E, K_1\otimes V)$, this means that $K_1$ is always $\phi$-invariant; a similar argument shows that $K_2$ is also $\phi$-invariant, implying that $(E, \phi)$ is stable if and only if $E$ is. This proves (a).
 
 Let us now assume that $E$ is not regular of the generic fibre of $\pi$ so that $K_1 = K_2$. In other words, $K = K_1 = K_2$ is the unique maximal destabilising bundle of $E$. If $E$ is an extension of line bundles, then $Z = \emptyset$ and the exact sequence \eqref{exact-filtrable-Higgs} becomes
 \begin{equation}\label{extension-O}
 0 \rightarrow \mc{O}_X \rightarrow K^{-1}\otimes E \rightarrow \pi^*H^{-1} \rightarrow 0
 \end{equation}
 after tensoring by $K^{-1}$.
Recall that extensions of the form \eqref{extension-O} are parameterised by $H^1(X,\pi^*H)$, and $H^1(X,\pi^*H)\cong H^0(B, H)\oplus H^1(B, H)$ by the Leray spectral sequence. Furthermore, any extension whose representative has non-zero first factor in this decomposition will be regular on the generic fibre, so those which are not regular on the generic fibre of $\pi$ are parameterised by $H^1(B, H)$. This group also parameterises extensions of $H^{-1}$ by $\mc{O}_B$ on $B$, and the pullback of such an extension is not regular on the generic fibre of $\pi$. Therefore, every extension of $\pi^*H^{-1}$ by $\mc{O}_X$ that is not regular on the generic fibre of $\pi$ is a pullback. Let $F$ be the rank-2 bundle on $B$ such that $K\otimes \pi^*F\cong E$. Then, 
\[ H^0(X,\End E\otimes V)=H^0(B,\pi_*(\End E)\otimes W)=H^0(B, \End F\otimes W),\]  
proving (b) (i).

Finally, let us assume that $E$ is not regular on the generic fibre of $\pi$ and is not an extension of line bundles. Then, $Z \neq \emptyset$ and $\pi_*(K^{-1}\otimes E)$ is a rank 2 vector bundle on $B$ given by an extension of the form 
\begin{equation}\label{non-reg-pushforward}0 \rightarrow \mc{O}_B \rightarrow \pi_*(K^{-1}\otimes E) \rightarrow L\rightarrow 0
\end{equation}
with $L = H^{-1} \otimes \pi_*(\mc{I}_Z) \in \Pic(B)$. By taking the tensor product of the exact sequence \eqref{exact-filtrable-Higgs} with $E^\vee\otimes V=\delta^{-1}\otimes E \otimes V$, we obtain 
\[ 0 \rightarrow \Hom(E, K \otimes V) \stackrel{i}{\rightarrow} \End(E) \otimes V \stackrel{p}{\rightarrow} K^{-1}\otimes E \otimes V \otimes \mc{I}_Z \rightarrow 0,\]
and as in the previous cases, the multiples of the identity in $\End(E)\otimes V$ are mapped injectively by $p$, so the space of trace-free Higgs fields contains a subspace isomorphic to 
\[ H^0(X, \Hom(E,K \otimes V))\simeq H^0(B, \pi_*(K^{-1}\otimes E)\otimes H\otimes W) = h^0(B, H\otimes W)\] 
by \eqref{non-reg-pushforward} because $h^0(B,\pi_*(\mc{I}_Z)) = h^0(X,\mc{I}_Z) = 0$ since $Z \neq \emptyset$. Therefore,
\[ h^0(X,\End_0(E)\otimes V)\geq h^0(B, \Hom(E,K \otimes V)) = h^0(B, H\otimes W).\]
In fact, $H^0(X,\End_0 E\otimes V) \simeq H^0(X,\Hom(E,K \otimes V)$ whenever $h^0(B,\pi_*(K^{-1} \otimes E) \otimes W) = h^0(B,W)$, in which case $K$ is $\phi$-invariant for all Higgs fields $\phi: E \rightarrow E \otimes V$, and Higgs fields are stable if and only if $E$ is. 
\end{proof}

Let us now consider the case where the underlying bundle $E$ in not filtrable. We first treat the case where the bundle has no jumps; we focus on the case where $E$ is regular on every fibre of $\pi$ as it is the only case necessary for our applications in sections \ref{Hopf} and \ref{geq2}. The general case can be analysed using the same techniques.

\begin{proposition}\label{non-filt-reg}
    Let $E$ be a rank-2 non-filtrable vector bundle on $X$ that is regular on every fibre of $\pi: X \rightarrow B$, and let $\rho:C \to B$ be the spectral curve of $E$. If $R$ is the ramification divisor of $\rho$, then $\pi_*(\End_0E)$ is a line bundle $N \in \Pic(B)$ such that $N^2\simeq \mc{O}_B(-R)$ and  $\deg N = -4\Delta(E)$, and $H^0(X,\End_0E\otimes \pi^*(W))\simeq H^0(B,N\otimes W)$ for any vector bundle $W$ on $B$.
\end{proposition}
\begin{proof}
    Suppose $E$ is a rank-2 non-filtrable bundle on $X$ that is regular on every fibre of $\pi$, and let $W$ be a vector bundle on $B$. This means, in particular, that $E$ is stable and so $h^0(X, \End_0(E))=0$. Moreover, on the general fibre $\pi^{-1}(b)$ of $\pi$, we have
\[ \End_0(E)\rvert_{\pi^{-1}(b)} = \mc{O}_{\pi^{-1}(b)} \oplus \lambda \oplus \lambda^{-1}\] 
for some $\lambda \in \Pic^0(\pi^{-1}(b))$ such that $\lambda^2 \neq \mathcal{O}_{\pi^{-1}(b)}$. Consequently, $\pi_*(\End_0(E))$ is a rank 1 torsion-free sheaf $N$ on $B$.
Moreover, $\pi_*(\End_0E\otimes \pi^*N^{-1})\simeq \mc{O}_B$, so if $\phi:E \to E\otimes N^{-1}$ is a map associated to a section $\tilde{\phi}\in H^0(X,\End_0E\otimes \pi^*N^{-1})$, then $\phi$ has positive rank on all fibres. Since $E$ is regular on all fibres $T$, the restriction $E\vert_T$ is either of the form $\lambda_1\oplus \lambda_2$ with $\lambda_1,\lambda_2 \in \Pic^0(T)$ and $\lambda_1\not\simeq \lambda_2$, or $\lambda\otimes \mc{A}$ with $\lambda \in \Pic^0(T)$ and $\mc{A}$ the Atiyah bundle. The fibres of the second type occur exactly at the ramification points of $\rho$. In both cases, the non-zero trace-free endomorphisms of $E\vert_T$ are uniquely determined up to a constant scaling, and $\phi\vert_F$ has rank 1 on the fibres corresponding to the ramification points of $\rho$ and rank 2 on all other fibres. Consider the exact sequence 
\[0 \rightarrow E \stackrel{\phi}{\rightarrow} E\otimes \pi^*N^{-1} \rightarrow \coker{\phi}\rightarrow 0.\]
Note that $\coker{\phi}$ is supported exactly on the fibres where the rank of $\phi$ is one, so the support of $\coker{\phi}$ is $\pi^{-1}(R)$. However, we also have that 
\[\det(E)\otimes \mc{O}_X(\supp(\coker(\phi)))\simeq \det(E\otimes \pi^*N^{-1})\simeq \det(E)\otimes \pi^* N^{-2},\]
implying that $\pi^* N^{-2}\simeq \mc{O}_X(\pi^{-1}(R))$ or, equivalently, that $N^2\simeq \mc{O}_B(-R)$. Since $\deg R = 8 \Delta(E)$, it follows that $\deg N = -4\Delta(E)$.

To compute $H^0(X,\End_0E\otimes \pi^*W)$, we now note that 
\[H^0(X,\End_0E\otimes \pi^*W)=H^0(B,\pi_*(\End_0E\otimes \pi^*W))=H^0(B,N\otimes W).\]
\end{proof}

Let us now examine non-filtrable bundles with jumps. We first prove the following technical lemma:

\begin{lemma}
 Let $E$ be a rank-2 non-filtrable vector bundle on $X$ that has a jump of multiplicity $\mu$ on the fibre $j: T = \pi^{-1}(b) \hookrightarrow X$ with $b \in B$.  Moreover, let $E'$ be the allowable elementary modification of $E$ along $T$. Then:
 \begin{enumerate}
     \item $\pi_*(\End_0 E) \simeq \pi_*(\End_0 E')$ if $E'\vert_T = E\vert_T$;
     \item $\pi_*(\End_0 E) \simeq \pi_*(\End_0 E')(-b)$ if $E'\vert_T \ne E\vert_T$.
 \end{enumerate}
\end{lemma}
\begin{proof}
Suppose that $E$ is a rank-2 non-filtrable bundle on $X$ that has a jump of multiplicity $\mu$ on the fibre $j: T = \pi^{-1}(b) \hookrightarrow X$ with $b \in B$. There then exists an exact sequence
\begin{equation}\label{allowable2}
0 \rightarrow E' \stackrel{\iota}{\rightarrow} E \stackrel{p}{\rightarrow} j_*L \rightarrow 0
\end{equation}
for some line bundle $L \in \Pic^d(T)$, $d < 0$. In other words, $E'$ is the allowable elementary modification of $E$. Tensoring \eqref{allowable2} by $E^*$, we obtain
\begin{equation}\label{starred}
0 \rightarrow \Hom(E,E') \stackrel{\iota_*}{\rightarrow} \End E \stackrel{p_*}{\rightarrow} \Hom(E,j_*L) \rightarrow 0.  
\end{equation}
Pushing forward equation \eqref{starred} to $B$ then gives
\[ 0 \rightarrow \pi_*(\Hom(E,E')) \rightarrow \pi_*(\End E) \rightarrow \mathcal{O}_b \rightarrow \dots .\]
Note that $\pi_*(\Hom(E,E'))$ is a rank-2 vector bundle on $B$ because $E$ and $E'$ are both non-filtrable. Similarly, 
\[ \pi_*(\End E) = \mathcal{O}_B \oplus \pi_*(\End_0 E) = \mathcal{O}_B \oplus N\]
with $N := \pi_*(\End_0 E)$ a line bundle on $B$.
Moreover, the factor $\mathcal{O}_B$ in $\pi_*(\End E)$ maps surjectively onto $\mathcal{O}_b \simeq \pi_*(\Hom(E,j_*L))$ by sending $\id_E$ to the projection $p: E \rightarrow j_* L$. We therefore have an exact sequence
\[ 0 \rightarrow \pi_*(\Hom(E,E')) \rightarrow \mathcal{O}_B \oplus N \rightarrow \mathcal{O}_b \rightarrow 0 .\]
This means, in particular, that $\pi_*(\Hom(E,E'))$ is an elementary modification of $\pi_*(\End E) = \mathcal{O}_B \oplus N$ along the divisor $D = b$ (or, equivalently, a Hecke modification of $\pi_*(\End E)$ at $b$), implying that
\[ \det(\pi_*(\Hom(E,E'))) = \det(\pi_*(\End E))(-D) = N(-b).\]
We further note that $\det E' = (\det E)(-T)$ so that
\[ \Hom(E,E') = E^* \otimes E' = E^* \otimes  (E')^* \otimes (\det E)(-T) = \Hom(E',E)(-T) \]
and
\[ \pi_*(\Hom(E,E')) = \pi_*(\Hom(E',E))(-b).\]
Hence,
\[ N = \det(\pi_*(\Hom(E,E')))(b) =  \det(\pi_*(\Hom(E',E)))(-b).\]

Let us now prove that $\det(\pi_*(\Hom(E',E)))$ can be expressed in terms of $\pi_*(\End_0 E')$.
Tensoring \eqref{allowable2} by $(E')^*$, we obtain
and
\begin{equation}\label{starred2}
0 \rightarrow \End(E') \stackrel{\iota_*}{\rightarrow} \Hom(E',E) \stackrel{p_*}{\rightarrow} \Hom(E',j_*L) \rightarrow 0.  
\end{equation}
Pushing down to $B$ gives the exact sequence
\begin{equation}\label{starred3} 
0 \rightarrow \pi_*(\End E') \rightarrow \pi_*(\Hom(E',E)) \rightarrow \pi_*(\Hom(E',j_*L)). 
\end{equation}
Note that 
\[ \pi_*(\End E') = \mathcal{O}_B \oplus \pi_*(\End_0 E') = \mathcal{O}_B \oplus N'\]
with $N' := \pi_*(\End_0 E')$ a line bundle on $B$ since $E'$ is non-filtrable.
Furthermore, $\pi_*(\Hom(E',j_*L)) = 0$ if $E'\vert_T \not\simeq E\vert_T$, in which case
\[ \pi_*(\End E') \simeq \pi_*(\Hom(E',E))\]
and 
\[ N = \det(\pi_*(\Hom(E',E)))(-b) =\det(\pi_*(\End E'))(-b) = N'(-b).\]
Whereas, if $E'\vert_T \simeq E\vert_T$, then $\pi_*(\Hom(E',j_*L)) = \mathcal{O}_b$ and equation \eqref{starred3} becomes
\begin{equation}\label{starred4} 
0 \rightarrow \pi_*(\End E') \rightarrow \pi_*(\Hom(E',E)) \rightarrow \mathcal{O}_b \dots . 
\end{equation}
The map of sheaves $\pi_*(\Hom(E',E)) \rightarrow \mathcal{O}_b$ is, however, surjective because 
\begin{align*} \pi_*(\Hom(E',E))_b &= H^0(T, \End(L' \oplus L)) \\
&= H^0(T, \End L') \oplus H^0(T, \Hom(L, L')) \oplus H^0(T, \End L),
\end{align*}
where $E'\vert_T \simeq L' \oplus L$ with $L' \in \Pic^{-d}(T)$,
and elements of $H^0(T, \Hom(L,L))$ map onto $\mathcal{O}_b \simeq \pi_*(\Hom(E',j_*L))$ under $p_*$. Consequently, equation \eqref{starred3} is the short exact sequence, implying that $\pi_*(\End E')$ is a Hecke transformation of $\pi_*(\Hom(E',E))$ at $b$. Thus,
\[ \det(\pi_*(\Hom(E',E))) = \pi_*(\End E')(b) = N'(b)\]
and 
\[ N = \det(\pi_*(\Hom(E',E)))(-b) = N',\]
proving the lemma.
\end{proof}

As a direct consequence of the lemma, we have:

\begin{proposition}
 Let $E$ be a rank-2 non-filtrable bundle on $X$ that has a jump of multiplicity $\mu$ and length $l$ on the fibre $j: T = \pi^{-1}(b) \hookrightarrow X$ with $b \in B$. Suppose that
 $\{ \bar{E}_1,\bar{E}_2, \dots, \bar{E}_l \}$ is the sequence of allowable elementary modifications associated to $E$, and let 
\[ h_1 \geq h_2 \geq \cdots \geq h_{l-1}\]
be the respective heights of $\bar{E}_1,\dots,\bar{E}_{l-1}$. 
Then, $\Delta(E) = \Delta(\bar{E}_l) + \frac{1}{2}\mu$ and
\[ \pi_*(\End_0 E) \simeq \pi_*(\End_0 \bar{E}_l)(-sb),\]
where $s$ is the number of distinct heights. 
\end{proposition}
\begin{proof}
Referring to the lemma, we have $\pi_*(\End_0 E) \simeq \pi_*(\End_0 \bar{E}_1)$ if
$h_0 = h_1$, and $\pi_*(\End_0 E) \simeq \pi_*(\End_0 \bar{E}_1)(-b)$ if  $h_0 > h_1$. Similarly, the lemma gives for any $1\leq i\leq l-1$ that $\pi_*(\End_0 \bar{E}_{i}) \simeq \pi_*(\End_0 \bar{E}_{i+1})$ if
$h_i = h_{i+1}$, and $\pi_*(\End_0 \bar{E}_i) \simeq \pi_*(\End_0 \bar{E}_{i+1})(-b)$ if  $h_i > h_{i+1}$. Applying this result iteratively to each of the $\bar{E}_i$ gives the statement of the proposition.
\end{proof}

\begin{corollary}\label{non-filt-jumps}
Let $E$ be a rank-2 vector bundle on $X$ that has $d$ jumps of multiplicity $\mu_i$ at the fibres $T_i = \pi^{-1}(b_i)$, where $b_i \in B$, $i = 1,\dots,d$. Suppose that the jump at $T_i$ has $s_i$ distinct heights. If $\bar{E}$ is the vector bundle obtained by removing all the jumps of $E$, then
\[ \Delta(E) = \Delta(\bar{E}) + \frac{1}{2}\sum_{i=1}^d\mu_i, \]
\[ \pi_*(\End_0 E) = \pi_*(\End_0 \bar{E})(-s_1b_1-\dots-s_db_d) \]
and
\[ H^0(X,\End_0(E)\otimes K_X) \simeq  H^0(X,\End_0(\bar{E})\otimes K_X(-s_1T_1-\dots-s_dT_d)).\]
\end{corollary}


\section{Co-Higgs bundles in the Hopf surface case}\label{Hopf}
In this section, $X$ is a Hopf surface so that $B = \mathbb{P}^1$ and $\mc{T}_X\cong \mc{O}_X(T)\oplus \mc{O}_X(T)$ with $T$ a fibre of $\pi$. A vector bundle $E$ on $X$ then admits a non-trivial co-Higgs bundle if and only if $H^0(X, \End_0(E)\otimes \mc{O}_X(T))\neq 0$. In fact, we can explicitly parameterise the Higgs fields on $E$ in terms of $H^0(X,\End_0(E)\otimes \mc{O}_X(T))$ as follows:
\begin{lemma}\label{hopf-integrability}
Let $X$ be a Hopf surface and let $E$ be a rank-2 vector bundle on $X$. The set of trace-free Higgs fields $\phi:E \to E\otimes \mc{T}_X$ is then
\[ \left\{\phi=(a\phi_0,b\phi_0) 
: \phi_0 \in H^0(X,\End_0(E)\otimes \mc{O}_X(F)), a,b \in \mc{M}_X(X)\right\},\]
where $\mc{M}_X$ is the sheaf of meromorphic functions on $X$.
\end{lemma}
\begin{proof}
Let $U\subseteq X$ be an open set that simultaneously trivialises $\mc{O}_X(T)$ and $E$. Let $e$ be a nowhere vanishing section of $\mc{O}_X(T)\rvert_U$ and let $\{e_1,e_2\}$ be the induced local frame on $\mc{T}_X=\mc{O}_X(T)\oplus \mc{O}_X(T)$. Then $\phi\rvert_U$ has the form $$\phi\rvert_U=\begin{bmatrix}a_1 & b_1 \\ c_1 & -a_1\end{bmatrix}\otimes e_1 + \begin{bmatrix}a_2 & b_2\\ c_2 & -a_2\end{bmatrix}\otimes e_2,$$
and $\phi\wedge\phi\rvert_U$ is $$\left(\phi\wedge\phi\right)\rvert_U=\begin{bmatrix}b_1c_2-b_2c_1 & 2(a_1b_2-a_2b_1)\\ 2(c_1a_2-c_2a_1) & c_1b_2-c_2b_1\end{bmatrix}\otimes e_1\wedge e_2.$$
Since $e_1\wedge e_2$ is nowhere-vanishing on $U$, we must have $a_1b_2-a_2b_1=a_1c_2-a_2c_1=b_1c_2-b_2c_1=0$. The result holds trivially if $\phi\rvert_U=0$, so assume that there is an open set $V\subseteq U$ so that one of the entries is non-zero, say $a_1\neq 0$. Then $b_2=\frac{a_2b_1}{a_1}$ and $c_2=\frac{a_2c_1}{a_1}$, so $\phi\rvert_V=\phi_0\otimes e_1 +\frac{a_2}{a_1} \phi_0 \otimes e_2$ where $\phi_0=\begin{bmatrix}a_1 & b_1\\ c_1 & -a_1\end{bmatrix}.$ We obtain a similar result by assuming any of the $a_i,b_i,c_i$ is non-zero. Since $\phi$ is defined globally, we can glue these local representations to get that there are meromorphic functions $a,b$ so that $\phi=(a\phi_0,b\phi_0)$ for some $\phi_0\in H^0(X,\End_0(E)\otimes \mc{O}_X(T))$.
\end{proof}

\begin{remark}
Note that there are examples of co-Higgs bundles on the Hopf surface $X$ where the meromorphic functions $a,b$ in Lemma \ref{hopf-integrability} are forced to have non-trivial singularities. For instance, consider the Higgs field on $E = \mc{O}_X\oplus \mc{O}_X$ given by $$\left(\begin{bmatrix}0 & x\\ 0 & 0\end{bmatrix}, \begin{bmatrix}0 & y\\ 0 & 0\end{bmatrix}\right),$$
where $x,y$ are homogeneous coordinates for $\mb{P}^1$ and we use the identification of sections of $\mc{O}_X(T)$ with degree-1 homogeneous polynomials in $x,y$. It can be easily verified that this Higgs field is integrable, but for any representation $(a\phi_0, b\phi_0)$, at least one of $a$ and $b$ will have a singularity.
\end{remark}

Before giving a complete description of rank-2 co-Higgs on Hopf surfaces, we begin with two examples that are important in the classification of co-Higgs bundles with trivial Chern classes.

\begin{example}\label{stable-co-Higgs-even}
Let us first consider the case where $E = \mc{O}_X \oplus \mc{O}_X$ so that 
\[ H^0(X,\End_0 E(F)) = H^0(X,\mc{O}_X(F))^{\oplus 3} \simeq H^0(\mathbb{P}^1,\mc{O}_{\mathbb{P}^1}(1))^{\oplus 3}. \] 
Let $x,y$ be homogeneous coordinates on $\mathbb{P}^1$. Then, 
\[  H^0(\mathbb{P}^1,\mc{O}_{\mathbb{P}^1}(1)) \simeq \{ ax+by : a,b \in \mathbb{C} \}\]
is the set of all homogeneous polynomials of degree 1 in $x$ and $y$.
Moreover, any element $\phi_0 \in H^0(X,\End_0 E(T))$ can be written as $\phi_0 = A_1x + A_2y$ for some $A_1,A_2 \in \mathfrak{sl}(2,\mathbb{C})$. Note that $(E,\phi_0)$ is stable if and only if $\mc{O}_X$ is not an invariant subbundle of $E$, which is equivalent to the constant matrices $A_1$ and $A_2$ not having a common eigenvector. Moreover, $A_1$ and $A_2$ do not have a common eigenvector if and only if $\det([A_1,A_2]) \neq 0$ (see \cite{Shemesh}, Theorem 3.1). We are, however, only considering $\phi_0$ up to conjugation by an element of $\Aut(E) \simeq \mathrm{GL}(2,\mathbb{C})$. Since $A_1$ is conjugate to the matrix
\[ \begin{bmatrix} t & 1 \\ 0 & -t \end{bmatrix},\]
where $t$ and $-t$ are the eigenvalues of $A_1$,  we can thus assume from the start that
\[ \mbox{$A_1 = \begin{bmatrix} t & 1 \\ 0 & -t \end{bmatrix}$ and $A_2 = \begin{bmatrix} u & v \\ s & -u \end{bmatrix}$} \]
for some $s,t,u, v \in \mathbb{C}$. Furthermore, $s \neq 0$ since $\det([A_1,A_2]) \neq 0$. Conjugating both $A_1$ and $A_2$ by
\[ \begin{bmatrix} 1 & -u/s \\ 0 & 1/s \end{bmatrix},\]
we obtain
\[ \mbox{$A_1' = \begin{bmatrix} t & s' \\ 0 & -t \end{bmatrix}$ and $A_2' = \begin{bmatrix} 0 & v' \\ 1 & 0 \end{bmatrix}$}, \]
where $s' = s+2tu$ and $v' = u^2+sv$, and $4t^2v' - (s')^2 = \det([A_1,A_2]) \neq 0$. Note that whenever $4t^2v' - (s')^2 \neq 0$, conjugation by
\[ \begin{bmatrix} s & 2tv\\ 2t & s \end{bmatrix}\]
swaps $t$ and $-t$ in $A_1$ and leaves $A_2$ unchanged. Consequently, the conjugacy class of $(A_1,A_2)$ is completely determined by the pair of eigenvalues $\pm t$ as well as the constants $s',v'$. In other words, the stable Higgs fields $\phi_0$ are parametrised by the GIT quotient
\begin{align*}S&= \Spec \left(\mathbb{C}[s',t, v',w]/(w(4t^2v' - (s')^2) - 1)\right)//(t \mapsto -t)\\
&\simeq\Spec \left(\mathbb{C}[s',t',v',w]/(w(4t'v'-(s')^2)-1)\right).\end{align*}
More concretely, such stable Higgs fields are of the form
\[ \phi_0 = A_1'x+A_2'y = \begin{bmatrix} tx & s'x + v'y \\ y & -tx \end{bmatrix} \]
for some $s',t,v' \in \mathbb{C}$ such that $4t^2v' - (s')^2 \neq 0$.
This means, in particular, that $s'$ and $t$ cannot simultaneously be 0, implying that only constant multiples of $\phi_0$ are holomorphic sections of $\End_0(T)$. Stable integrable elements of $H^0(X,\End_0 E \otimes \mc{T}_X)$ are of the form $\phi = (a\phi_0,b\phi_0)$ with $(a,b) \in \mathbb{C}^2\setminus \{(0,0)\}$. Given that $a \neq 0$ or $b \neq 0$, we can absorb this non-zero constant in the eigenvalues of $A_1'$. Consequently, the moduli space of stable trace-free co-Higgs bundles with underlying bundle $E = \mc{O}_X \oplus \mc{O}_X$ is isomorphic to $S \times \mathbb{P}^1$.
\end{example}

\begin{example}\label{stable-co-Higgs-odd}
Let us now consider the case where $E = \mc{O}_X \oplus \mc{O}_X(-T)$ so that elements $\phi_0 \in H^0(X,\End_0 E(T))$ are of the form
\[ \phi_0 = \begin{bmatrix} a & b \\ c & -a \end{bmatrix}\]
with $a \in H^0(X,\mc{O}_X(T)), b \in H^0(X,\mc{O}_X(2T))$ and $c \in \mathbb{C}$, and automorphisms of $E$ are of the form
\[ \begin{bmatrix} d & e \\ 0 & f \end{bmatrix}\]
with $d,f \in \mathbb{C}^*$ and $e \in H^0(X,\mc{O}_X(T))$. Note that $\phi_0$ is stable if and only if $\mathcal{O}_X$ is not $\phi_0$-invariant. Since $\mc{O}_X$ can only map into the first factor of $E$, this happens if and only if $c \neq 0$. We are, however, only considering $\phi_0$ up to conjugation by an automorphism of $E$. Therefore, when $c \neq 0$, we can conjugate by
\[ \psi = \begin{bmatrix} 1 & -a/c \\ 0 & 1/c \end{bmatrix}\]
to obtain
\[ \phi_0' = \psi \phi_0 \psi^{-1} = \begin{bmatrix} 0 & b' \\ 1 & 0 \end{bmatrix},\]
where $b' = cb + a^2$. Consequently, the conjugacy class of $\phi_0$ is completely determined by $b' \in H^0(X,\mc{O}_X(2T))$. In other words, the stable Higgs fields are parametrised by 
\[ H^0(X,\mc{O}_X(2T)) = H^0(\mathbb{P}^1,\mc{O}_{\mathbb{P}^1}(2)) = \{ b_1x^2 + b_2xy + b_3y^2: b_1,b_2,b_3 \in \mathbb{C} \} \simeq  \mathbb{C}^3, \]
the set of all homogeneous polynomials of degree 2 in the homogeneous coordinates $x,y$ on $\mathbb{P}^1$.
More concretely, such stable Higgs fields are of the form
\[ \phi_0 = \begin{bmatrix} 0 & b_1x^2 + b_2xy + b_3y^2\\ 1 & 0 \end{bmatrix} \]
for some $b_1,b_2,b_3 \in \mathbb{C}$,
implying that only constant multiples of $\phi_0$ are holomorphic sections of $\End_0E(T)$. Stable integrable elements of $H^0(X,\End_0 E \otimes \mc{T}_X)$ are thus of the form $\phi = (a\phi_0,b\phi_0)$ with $(a,b) \in \mathbb{C}^2\setminus \{(0,0)\}$ and $\phi_0 \in H^0(X,\End_0 E(T))$. Given that $a \neq 0$ or $b \neq 0$, the non-zero constant $a$ or $b$ can also be absorbed in the constant $c$ of $\phi_0$. Consequently, the moduli space of stable trace-free co-Higgs bundles with underlying bundle $E = \mc{O}_X \oplus \mc{O}_X(-T)$ is isomorphic to $\mathbb{C}^3 \times \mathbb{P}^1$.
\end{example}

\begin{remark}\label{pullback-co-Higgs}
The stable co-Higgs bundles described in examples \ref{stable-co-Higgs-even} and \ref{stable-co-Higgs-odd} correspond to stable co-Higgs bundles on $\mathbb{P}^1$. 
To see this, recall from Proposition \ref{cotangent} that $\mc{T}_X$ fits into the exact sequence
\[ 0 \rightarrow  \mc{O}_X \rightarrow \mc{T}_X \rightarrow \pi^* \mc{T}_{\mathbb{P}^1} \rightarrow 0.\]
Tensoring this exact sequence by $\End_0 E$ for any vector bundle $E$ on $X$ and passing to cohomology, we obtain the long-exact sequence
\[ 0 \rightarrow H^0(X,\End_0 E) \rightarrow H^0(X,\End_0 E \otimes \mc{T}_X) \rightarrow H^0(X,\End_0 E \otimes\pi^* \mc{T}_{\mathbb{P}^1}) \dots. \]
In the case where $E = \mc{O}_X \oplus \mc{O}_X$ or $\mc{O}_X \oplus \mc{O}_X(-T)$, a direct computation gives the short exact sequence:
\[ 0 \rightarrow H^0(X,\End_0 E) \rightarrow H^0(X,\End_0 E \otimes \mc{T}_X) \rightarrow H^0(\mathbb{P}^1,\End_0 V \otimes \mc{T}_{\mathbb{P}^1}) \rightarrow 0, \]
where $V = \mc{O}_{\mathbb{P}^1} \oplus \mc{O}_{\mathbb{P}^1}$ or $\mc{O}_{\mathbb{P}^1} \oplus \mc{O}_{\mathbb{P}^1}(-1)$ depending on whether $E = \mc{O}_X \oplus \mc{O}_X$ or $\mc{O}_X \oplus \mc{O}_X(-T)$, respectively. Note that stable Higgs fields on $X$ cannot come from elements in $H^0(X,\End_0 E)$, otherwise, there would exist a $\phi \in H^0(X,\End_0 E)$ with $(E,\phi)$ stable, which would force $\phi = \lambda \id_E$ for some $\lambda \in \mathbb{C}$ by Proposition \ref{endo}. Since $\phi$ is assumed to be trace-free, this means that $\lambda = 0$ and $\phi = 0$, which is impossible since $E$ is not stable. Every stable Higgs field on $E$ is thus the pullback of a stable Higgs field on $V$, whose moduli spaces were studied in \cite{Rayan,Rayan-P1}. In fact, it was shown that the component corresponding to $V = \mc{O}_{\mathbb{P}^1} \oplus \mc{O}_{\mathbb{P}^1}(-1)$ is a 6-dimensional subvariety of $M \times H^0(\mathbb{P}^1,\mc{O}_{\mathbb{P}^1}(4))$, where $M$ is the total space of $\mc{T}_{\mathbb{P}^1}$. Nonetheless, not every stable Higgs field on $V$ gives rise to a stable Higgs field on $E$ because of the integrability condition $\phi \wedge \phi = 0$ as described in examples \ref{stable-co-Higgs-even} and \ref{stable-co-Higgs-odd}. As such, the stable Higgs fields on $E$ form a proper subset of the stable Higgs fields on $V$.
\end{remark}

Let us now give a complete description of rank-2 co-Higgs bundles $(E,\phi: E \rightarrow E \otimes \mc{T}_X)$ on $X$. We consider the cases where $E$ is not filtrable or filtrable separately. We first assume that $E$ is not filtrable.

\begin{proposition}\label{non-filt-co-Higgs}
Let $E$ be a rank-2 holomorphic vector bundle  on a Hopf surface $X$ that is not filtrable. Then, $\phi = 0$ is the only trace-free Higgs field on $E$.
\end{proposition}
\begin{proof}
It is enough to show that $h^0(X, \End_0(E(T)))=0$, where $T=\pi^{-1}\{b\}$ for some $b \in B$, where $\pi:X\to \mathbb{P}^1$ is the natural projection.
Let $\delta = \det E$ and set $N:=\pi_*(\End_0(E))$.
Since the Néron-Severi group of a Hopf surface is trivial, the ruled surface $\mathbb{F}_\delta$ is exactly $\mathbb{P}^1\times \mathbb{P}^1$. Because of this, every non-filtrable rank-2 bundle has positive second Chern class $c_2$ and discriminant $\Delta(2,\delta,c_2)=c_2/2$. Note that by Corollary \ref{non-filt-jumps}, if there is a non-filtrable bundle $E$ such that $h^0(X,\End_0(E(T)))\neq 0$, then there is a regular bundle $E'$ obtained by applying a series of allowable elementary modifications to $E$, as well as an effective divisor $D$ on $B$ so that $h^0(B,N(b+D))=h^0(X,\End_0(E'(T)))$. Thus it will suffice to consider the case where $E$ is regular. In this case, $\deg(N(b))=-4\Delta(2,\delta,c_2)+1=1-2c_2$ by Proposition \ref{non-filt-reg}. Since $c_2$ is a positive integer, $N(b)$ will therefore always have negative degree, giving $h^0(X,\End_0(E(T)))=0$.

\end{proof} 

Proposition \ref{non-filt-co-Higgs} therefore tells us that only filtrable bundles potentially admit non-trivial Higgs fields. We now consider the case where $E$ is filtrable, beginning with Higgs fields in $H^0(X,\End_0 E(T))$.

\begin{proposition}\label{hopf-reg}
Let $E$ be a filtrable rank-2 bundle over $X$ with $c_2(E) = c_2$ and maximal destabilising bundles $K_1, K_2$. Set $m:=\deg(K_1\otimes K_2\otimes \det(E)^{-1})$.
\begin{enumerate}
\item[(a)]
If $E$ is regular on the generic fibre of $\pi$, then 
\[ h^0(X,\End_0E(T)) = \max\{0,m+2\}.\]
In particular,
$H^0(X,\End_0E(T) \simeq H^0(X,\Hom(E,K(T)))$ so that
 $K_1$ and $K_2$ are both $\phi$-invariant for all $\phi \in H^0(X,\End_0E(T))$.

\item[(b)]
If $E$ is not regular on the generic fibre, so that $K_1 = K_2$ and $\pi_*(K_1^{-1}\otimes E)=\mc{O}_{\mb{P}^1}\oplus \mc{O}_{\mb{P}^1}(-\ell)$ for some $\ell\geq 0$ with $m \leq \ell$,
we then have two cases:
\begin{enumerate}
    \item[(i)] If $c_2 = 0$, then $m=\ell \geq 0$ and $E = K_1 \otimes (\mathcal{O}_X \oplus \mathcal{O}_X(-m T))$, so that
    \[ h^0(X,\End_0 E(T)) = \max\{ 6, m+4 \}.\]
In fact, $H^0(X,\End_0 E(T)) \simeq H^0(X,\Hom(E,K_1(T)))$ when $m \geq 2$, in which case $K_1$ is $\phi$-invariant for all $\phi \in H^0(X,\End_0E(T))$.

\item[(ii)] If $c_2 > 0$,  then $m < \ell$ and
\[h^0(X,\End_0E(T))= \max\{0,m+2\} + \max\{0,m-\ell+2\}.\]
In particular, $H^0(X, \End_0E(T)) \simeq H^0(X,\Hom(E,K_1(T)))$ so that $K_1$ is $\phi$-invariant for all $\phi \in H^0(X,\End_0E(T))$. 
\end{enumerate}

\end{enumerate}
\end{proposition}
\begin{proof}
The cases where $E$ is regular on the generic fibre and where $c_2=0$ follow from Proposition \ref{fitrable} with $V=\mc{O}_X(T)\simeq \pi^*(\mc{O}_{\mathbb{P}^1}(1))$ and $H=\mc{O}_{\mathbb{P}^1}(m)$.

In the case that $E$ is not regular on the generic fibre and has positive second Chern class $c_2$, the bundle fits into the exact sequence 
\begin{equation}\label{exact-not-reg}
 0 \rightarrow K_1 \stackrel{i}{\rightarrow} E \stackrel{p}{\rightarrow} K_1(-mT)\otimes \mc{I}_Z \rightarrow 0  
\end{equation}
with $Z$ a $0$-dimensional subspace of length $c_2$, and $\pi_*(K_1^{-1}\otimes E)\simeq \mc{O}_{\mb{P}^1}\oplus \mc{O}_{\mb{P}^1}(-\ell).$ Note that in this case we necessarily have $\ell>m$. In this case, $\ell > m$. We again study $H^0(X,\End_0(E)(T))$ via the left exact sequence
\[0 \rightarrow H^0(X,\Hom(E,K_1(T))) \rightarrow H^0(X,\End(E)(T)) \rightarrow H^0(X, K_1^{-1}\otimes E(T) \otimes \mc{I}_Z) \dots.\]
Under these assumptions, we have
 \begin{align*} 
 h^0(X,\Hom(E,K_1(T))) &= h^0(X,K_1^{-1} \otimes E((m+1)T))\\
&= h^0(\mathbb{P}^1,\mathcal{O}_{\mathbb{P}^1}(m+1) \oplus 
 \mathcal{O}_{\mathbb{P}^1}(m+1-\ell))\\
 &= \max\{0,m+2\} + \max\{0,m-\ell+2\}
 \end{align*}
and 
\[ h^0(X,K_1^{-1} \otimes E(T)) = h^0(\mathbb{P}^1,\mathcal{O}_{\mathbb{P}^1}(1) \oplus 
 \mathcal{O}_{\mathbb{P}^1}(1-\ell)) = 2 + h^0(\mathbb{P}^1, \mathcal{O}_{\mathbb{P}^1}(1-\ell)).\]
 We consider three cases: $\ell \geq 2$, $
\ell =1$ and $\ell = 0$. 

 If $\ell \geq 2$, then $h^0(\mathbb{P}^1, \mathcal{O}_{\mathbb{P}^1}(1-\ell)) = 0$ and $h^0(X,K^{-1} \otimes E(T)) = 2$.
 This then means that 
 \[H^0(X,K^{-1} \otimes E(T) \otimes I_Z) = (p \otimes \id_{\mathcal{O}(T)})_*(H^0(X,\mathcal{O}_X(T))),\]
 and the elements of this space correspond to Higgs fields of the form 
 \begin{equation}\label{trace-Higgs} 
 \phi = \id_E \otimes s
 \end{equation}
 with $s \in H^0(X,\mathcal{O}_X(T))$.
 The exact sequence \eqref{exact-not-reg} thus yields \[H^0(X,\End_0 E (T)) \simeq H^0(X,\Hom(E,K_1(T))).\] 

If $\ell=1$ and $c_2>0$, then $E$ has at least one jump, say over the fibre $T$. Denote by $E'$ the allowable elementary modification of $E$ along $T$. Then,  there is an exact sequence 
\begin{equation}\label{allowable-elementary-mod}
0 \rightarrow E' \rightarrow E \rightarrow j_*\lambda \rightarrow 0    
\end{equation}
with $j:T\to X$ the inclusion map and $\lambda$ a line bundle of negative degree on $T$. In particular, $\det E' = \delta(-T)$. Tensoring \eqref{allowable-elementary-mod} by $E^\vee(T)$, we obtain
\[ 0 \rightarrow E' \otimes E^\vee (T) \rightarrow E \otimes E^\vee (T) \rightarrow j_*\lambda \otimes E^\vee (T) \rightarrow 0. \]
Given that $E' \otimes E^\vee (T) = E' \otimes (E \otimes \delta^{-1})(T) = E' \otimes E \otimes (\det E')^{-1} \simeq \Hom(E',E)$, this reduces to
\[ 0 \rightarrow \Hom(E',E) \rightarrow \End E(T) \rightarrow j_*\lambda \otimes E^\vee (T) \rightarrow 0. \]
Taking the long exact sequence on cohomology yields
\[ 0 \rightarrow H^0(X,\Hom(E',E)) \rightarrow H^0(X,\End E(T)) \rightarrow H^0(X,j_*\lambda \otimes E^\vee (T)) \rightarrow \dots. \]
Note that \begin{align*}H^0(X,j_*\lambda \otimes E^\vee (T)) \simeq H^0(T,\lambda \otimes (E^\vee(T))\vert_T) 
= H^0(F,(\lambda^2\otimes (\delta\vert_T)^{-1}) \oplus \mathcal{O}_T),\end{align*} 
implying that $h^0(X,j_*\lambda \otimes E^\vee (T)) = 1$ because $\deg \lambda < 0$ and $\deg(\delta\vert_T) = 0$.
Hence, 
\[ h^0(X,\End E(T)) \leq h^0(X,\Hom(E',E)) +1.\]
In addition, $K_1$ is the unique maximal destabilising bundle of $E'$. Consequently, $E'$ fits into an exact sequence

\[0 \rightarrow K_1 \rightarrow E' \rightarrow K_1(-(m+1)T)\otimes I_{Z'} \rightarrow 0\]
with $Z'$ a subset of $Z$ (counting multiplicity), since $K_1^{-1}\otimes \det E'\simeq K_1(-(m+1)T)$. Applying the $\Hom(-,E)$ functor to this exact sequence, we obtain 
\[0 \rightarrow \Hom(E,K_1(T)) \rightarrow \Hom(E',E) \rightarrow K_1^{-1}\otimes E\otimes I_{Z'} \rightarrow 0\]
because $K_1^{-1}\otimes E((m+1)T) \simeq \Hom(E,K_1(T))$.
Taking the long exact sequence on cohomology gives the left-exact sequence
\[0 \rightarrow H^0(X,\Hom(E,K_1(T))) \rightarrow H^0(X,\Hom(E',E)) \rightarrow H^0(X,K_1^{-1} \otimes E \otimes I_{Z'}).\]
Note that 
$h^0(X,K_1^{-1} \otimes E \otimes I_{Z'}) \leq h^0(X,K_1^{-1} \otimes E) =h^0(\mathbb{P}^1, \mc{O}_{\mathbb{P}^1}\oplus\mc{O}_{\mathbb{P}^1}(-1)) = 1$ so that
\[ h^0(X,\Hom(E',E))\leq h^0(X, \Hom(E,K_1(T))+1.\]
Given that $h^0(X,\End E(T)) \leq h^0(X,\Hom(E',E))+1$ and $H^0(X,\Hom(E,K(T))$ maps injectively into $H^0(X,\End E(T))$ via $(i \otimes \id_{\mathcal{O}(T)})_*$, we then have \[ h^0(X, \Hom(E, K_1(T))) \leq h^0(X,\End E(T)) \leq h^0(X, \Hom(E, K_1(T)))+2.\] 
However, $H^0(X,\End E(T))$ contains sections of the form \eqref{trace-Higgs}, none of which are elements of $(i \otimes \id_{\mathcal{O}(T)})_*\left(H^0(X,\Hom(E,K_1(T))\right)$, implying that $h^0(X,\End E(T))=h^0(X, \Hom(E, K_1(T)))+2$ and $H^0(X,\End_0 E(T)) \simeq H^0(X, \Hom(E, K_1(T)))$.


Finally, if $l = 0$, then $m < 0$, implying that $E$ is stable and $h^0(X,\End E) = 1$. In addition, $\pi_*(K_1^{-1} \otimes E) = \mathcal{O}_{\mathbb{P}^1} \oplus 
 \mathcal{O}_{\mathbb{P}^1}$, which means that the injections $\iota: K_1 \rightarrow E$ are parametrised by 
 \[ H^0(X,K_1^{-1} \otimes E) = H^0(\mathbb{P}^1, \pi_*(K_1^{-1} \otimes E)) = H^0(\mathbb{P}^1,\mathcal{O}_{\mathbb{P}^1} \oplus 
 \mathcal{O}_{\mathbb{P}^1}) = \mathbb{C}^2.\]
 Moreover,
\begin{align*} 
H^0(X,K_1^{-1} \otimes E(T)) &= H^0(\mathbb{P}^1,\mathcal{O}_{\mathbb{P}^1}(1) \oplus 
 \mathcal{O}_{\mathbb{P}^1}(1)) = 
 H^0(X,K_1^{-1} \otimes E) \otimes H^0(X,\mathcal{O}_X(T))\\
 &= \{ \iota \otimes s: \mbox{$\iota:K_1 \rightarrow E$ is an injection}, s \in H^0(X,\mathcal{O}_X(T) \} \\ 
 &= \{ \iota^*(\id_E \otimes s): \mbox{$\iota:K_1 \rightarrow E$ is an injection}, s \in H^0(X,\mathcal{O}_X(T) \}.
 \end{align*}
 Consequently, none of the non-zero elements in $H^0(X,K_1^{-1} \otimes E(T))$ correspond to traceless Higgs fields.
Hence, $H^0(X,\End_0 E (T)) \simeq H^0(X,\Hom(E,K_1(T)))$.
\end{proof}

\begin{corollary}\label{hopf-filt-not-stable}
Let $E$ be a filtrable rank-2 bundle over $X$ with $c_2(E) = c_2$ and $\phi \in H^0(X,\End_0E(T))$. We have the following:
\begin{enumerate}
    \item[(a)] If $c_2 > 0$, then $(E,\phi)$ is stable if and only if $E$ is stable.
    \item[(b)] If $c_2= 0$ and $(E,\phi)$ is stable, then $E = K \oplus K$ or $K \oplus K(-T)$ for some $K \in \Pic X$.
\end{enumerate}
\end{corollary}
\begin{proof}
If $E$ is stable, then $c_2>0$ and $(E,\phi)$ is automatically stable. Let us therefore assume that $E$ is not stable with maximal destabilising bundles $K_1,K_2$. Then, $\mu(K_1) \geq \mu(E)$ or $\mu(K_2) \geq \mu(E)$. Without loss of generality, we can assume that $\mu(K) \geq \mu(E)$.
Suppose that $E$ has a stable Higgs field $\phi$. Then $K$ is not $\phi$-invariant, and this can only happen if $E = K \oplus K$ or $K \oplus K(-T)$ for some $K \in \Pic X$ by Proposition \ref{hopf-reg}.
\end{proof}

Putting all this together, we obtain the following:

\begin{theorem}\label{stable-co-Higgs}
Let $X$ be a Hopf surface and $c_2$ be a non-negative integer. Moreover, let $(E,\phi)$ be a trace-free rank-2 co-Higgs bundle on $X$ with $c_2(E) = c_2$. We have:
\begin{enumerate}
    \item If $E$ is non-filtrable, in which case $c_2 > 0$, then $\phi = 0$.
    \item If $E$ is filtrable, then one can have $\phi \neq 0$ for all $c_2 \geq 0$. 
    \item If $c_2 > 0$, then $(E,\phi)$ is stable if and only if $E$ is stable.
    \item If $c_2 = 0$, then $E$ admits a non-trivial stable Higgs field if and only if $E = L \oplus L$ or $L \oplus L(-T)$ for some $L \in \Pic(X)$.
\end{enumerate}
Consequently, there exists a non-trivial stable rank-2 co-Higgs bundle $(E,\phi)$ on $X$ with $c_2(E) = c_2$ for all $c_2 \geq 0$. 
\end{theorem}

\begin{proof}
Part 1 follows immediately from Proposition \ref{non-filt-co-Higgs}, so we can restrict to the case where $E$ is filtrable. We thus assume that $E$ is filtrable. Then, there exist non-trivial Higgs fields $\phi$ for all $E$ with $c_2 \geq 0$ by Proposition \ref{hopf-reg}. In the case that $c_2>0$, Corollary \ref{hopf-filt-not-stable} imply that $(E,\phi)$ is stable if and only if $E$ is stable. We are therefore left with considering the case where $c_2 = 0$. Referring to Corollary \ref{hopf-filt-not-stable}, if $c_2=0$ and $E$ admits a stable Higgs field then, $E = L \oplus L$ or $L \oplus L(-T)$ for some $L \in \Pic X$.

To construct a stable co-Higgs bundle with second Chern class $c_2>0$, let $L_1$ and $L_2$ be distinct line bundles in $\Pic(X)$ such that $-1<\deg(L_1)-\deg(L_2)<1$. This means, in particular, that $L_2 \otimes L_1^{-1} \neq \mathcal{O}_X(\ell T)$ for all $\ell \in \mathbb{Z}$. Set $E'=L_1\oplus L_2$ and let $E$ be an elementary modification of $E'$ along a fibre $T$ of $\pi$ by a line bundle $\lambda$ on $F$ of degree $c_2$. Then, $\det(E)=\det(E')\otimes \mc{O}_X(-T)$, $c_2(E)=c_2$, and since $\lambda$ has positive degree, the destabilising line bundles of $E$ are $L_1(-T)$ and $L_2(-T)$. We now have $\mu(E)=\frac{\deg(L_1)+\deg(L_2)-1}{2}$, and since $-1<\deg(L_1)-\deg(L_2)<1$,
\[ \deg(L_i(-T))=\deg(L_i)-1<\frac{\deg(L_1)+\deg(L_2)}{2}-\frac{1}{2} = \mu(E), \]
$i = 1,2$, so $E$ is stable. Furthermore, 
\[m=\deg\left(L_1(-T)\otimes L_2(-T) \otimes (L_1\otimes L_2\otimes \mc{O}_X(-T))^{-1}\right)=-1.\] 
Note that $E$ is regular on the generic fibre of $\pi$ since $L_1\vert_{\pi^{-1}(b)} \neq L_2\vert_{\pi^{-1}(b)}$ for all $b \in \mathbb{P}^1$. Hence, $h^0(X,\End_0 E(T)) =1$ by Proposition \ref{hopf-reg} (a), implying that every trace-free Higgs field $\phi: E \rightarrow E(T)$ is of the form $a\phi_0$ for some $a \in \mathbb{C}$ and a fixed $\phi_0 \in H^0(X,\End_0 E(T)$. Consequently, $h^0(X,\End_0 E \otimes \mathcal{T}_X) = 2$ and $\phi \wedge \phi = 0$ for all $\phi \in H^0(X,\End_0 E \otimes \mathcal{T}_X)$ giving us a two-dimensional family of trace-free Higgs fields on $E$. 

Finally, let us assume that $c_2 = 0$. Referring to Corollary \ref{hopf-filt-not-stable}, if $E$ admits a stable Higgs, then $E = K \oplus K$ or $K \oplus K(T)$ for some $K \in \Pic(X)$. Let us verify that it does admit stable Higgs fields in both cases. It is enough to do it for $E = \mc{O}_X\oplus \mc{O}_X$ or $\mc{O}_X\oplus \mc{O}_X(-T)$, and we know that stable Higgs fields exist for both by examples \ref{stable-co-Higgs-even} and \ref{stable-co-Higgs-odd}.
\end{proof}

In the case where $c_2=0$, we have a complete description of the moduli spaces of co-Higgs bundles on $X$:
\begin{corollary}\label{moduli-co-Higgs}
Let $\mathcal{M}^{st}_{coH,0}(X)$ be the moduli space of stable trace-free rank-2 co-Higgs bundles with second Chern class $c_2=0$ on the Hopf surface $X$. Then, $\mathcal{M}^{st}_{coH,0}(X)$ is a 5-dimensional variety with two disjoint components $\mathcal{Z}_1$ and $\mathcal{Z}_2$ given by
\[ \mathcal{Z}_1 \simeq \Pic(X) \times \Spec \left(\mathbb{C}[s,z, v,w]/(w(4zv - s^2) - 1)\right) \times \mathbb{P}^1\]
and
\[ \mathcal{Z}_2 \simeq \Pic(X) \times \mathbb{C}^3 \times \mathbb{P}^1.\]
Moreover, $\mathcal{M}^{st}_{coH,0}(X)$ is a codimension-1 subvariety of the moduli space of co-Higgs bundles on $\mathbb{P}^1$.
\end{corollary}

\begin{proof}
 This is a direct consequence of examples \ref{stable-co-Higgs-even} and \ref{stable-co-Higgs-odd}, Theorem \ref{stable-co-Higgs} and remark \ref{pullback-co-Higgs}.   
\end{proof}

\section{Higgs bundles in the Kodaira dimension 1 case}\label{geq2}

Let $\pi:X\to B$ be a non-K\"ahler principal elliptic surface with base $B$ of genus $g\geq 2$. In this case, $\deg(K_X)=2g-2>0$. Note that for any trace-free stable Higgs bundle $(E, \phi)$ on $X$, the Higgs field $\phi$ is of the form $(\id_E\otimes i)\circ\vp$ for some trace-free $\vp:E \to E\otimes K_X$ with $(E,\vp)$ stable by Proposition \ref{extension}, where $i:K_X \to \mc{T}_X^*$ is the injection given in Proposition \ref{cotangent}. One can easily verify that $((\id_E\otimes i)\circ\vp) \wedge ((\id_E\otimes i)\circ\vp)=0$ for any $\vp \in H^0(\End_0(E)\otimes K_X)$, so it suffices to study trace-free stable pairs of the form $(E, \vp:E \to E \otimes K_X)$, that is, Vafa--Witten pairs.

We divide our analysis in two parts by first studying filtrable bundles and then non-filtrable ones.

\begin{proposition}\label{regular}
Let $E$ be a rank-2 filtrable vector bundle on $X$, and let $K_1$ and $K_2$ be its maximal destabilising bundles. Set $H:=\pi_*(\det(E)^{-1}\otimes K_1\otimes K_2)$.
\begin{enumerate}
    \item[(a)] Suppose that $E$ is regular on the generic fibre of $\pi$. Then,
    \[ H^0(X,\End_0(E)\otimes V) \simeq H^0(B, H \otimes K_B),\]
    and a trace-free Higgs field $\phi: E \rightarrow E \otimes K_X$ is stable if and only if $E$ is.
    
    \item[(b)] Suppose that $E$ is not regular on the generic fibre of $\pi$ so that $K = K_1 = K_2$ is its unique maximal destabilising bundle. We have two cases:
    \begin{enumerate}
        \item[(i)] 
        If $E$ is an extension of line bundles, then $E\simeq K \otimes \pi^*(F)$ for some rank-2 vector bundle $F$ on $B$ that is an extension of $H^{-1}$ by $\mathcal{O}_B$. Furthermore, 
        \[ H^0(X,\End E\otimes K_X)= H^0(B, \End F\otimes K_B)\] 
        so that Higgs fields on $E$ are pullbacks of Higgs fields on $F$.
        
        \item[(ii)] If $E$ is not an extension of line bundles, then 
        \[ h^0(X,\End_0(E)\otimes K_X)\geq h^0(X,\Hom(E,K \otimes K_X)) = h^0(B,H\otimes K_B).\] 
        In fact, $H^0(X,\End_0 E\otimes K_X) \simeq H^0(X,\Hom(E,K \otimes K_X))$ whenever $h^0(B,K_B) = h^0(B,\pi_*(K^{-1} \otimes E) \otimes K_B)$, in which case $K$ is $\phi$-invariant for all Higgs fields $\phi: E \rightarrow E \otimes K_X$, and Higgs fields are stable if and only if $E$ is. 
    \end{enumerate}
\end{enumerate}

\end{proposition}

\begin{proof}
This is a direct consequence of Proposition \ref{fitrable} with $V = K_X = \pi^* K_B$.
\end{proof}

\begin{remark}
    Using similar arguments to those in the $\ell=1, c_2>0$ case of Proposition \ref{hopf-reg}, we can show that the bounds in (b) (ii) are sharp for a vector bundle $E$ if there is a point $b \in B$ such that $E$ has an allowable elementary modification along $\pi^{-1}(b)$ and $H^0(B, L\otimes K_B(-b))=0,$ where $L$ is as defined in the exact sequence \eqref{non-reg-pushforward}.
\end{remark}

By using the descriptions of the trace-free Higgs fields in the above cases, we can show that the Chern classes of a filtrable rank-2 bundle with a non-trivial Higgs field are only restricted in that they must be in the filtrable range.

\begin{proposition}\label{filt}
Let $E$ be a rank-2 filtrable bundle on $X$ with $c(E)=(1,c_1,c_2)$. Then, there is a stable rank-2 vector bundle $F$ on $X$ with $\det(F)=\det(E)$ and $c_2(F)=c_2(E)$ such that $F$ has a non-trivial trace-free Higgs field.
\end{proposition}

\begin{proof}
Since $E$ is filtrable, there are line bundles $L_1$ and $L_2$ and a finite set of points $Z$ on $X$ (counting multiplicity) such that $E$ fits into an exact sequence of the form
\[ 0 \rightarrow L_1 \rightarrow E \rightarrow L_2\otimes \mc{I}_Z \rightarrow 0. \]
In particular, $\det E = L_1 \otimes L_2$ and $c_2(E) = c_1(L_1) \cdot c_1(L_2) + \vert Z \vert$. Moreover, if $\Sigma_1$ and $\Sigma_2$ are the section of $J(X)$ corresponding to $L_1$ and $L_2$, respectively, then $(\Sigma_1 + \Sigma_2)$ is the spectral curve of $E$ when $Z = \emptyset$. Recall that $\Sigma_1 \cdot \Sigma_2 \neq 0$ if and only if $c_1(L_1 \otimes L_2^{-1})^2 \neq 0$.
We split the proof into three cases: 
\begin{itemize}\item[(i)] $Z \neq \emptyset$,
\item[(ii)]$Z = \emptyset$ and $\Sigma_1 \cdot \Sigma_2 \neq 0$, and 
\item[(iii)]$Z = \emptyset$ and $\Sigma_1 \cdot \Sigma_2 = 0$. 
\end{itemize}

We first consider the case where $Z\neq \emptyset$. Let $L_a$ be a line bundle on $X$ given by a constant factor of automorphy $a$ such that $$-1<\deg(L_1)-\deg(L_2)+2\deg(L_a)<0.$$ Since the degree function maps line bundles with constant factor of automorphy surjectively onto $\mathbb{R}$, such an $L_a$ always exists. Note that $\deg(L_1)-\deg(L_2)+2\deg(L_a)$ is not an integer. Therefore, $L_1 \otimes L_2^{-1} \otimes L_a^2$ is not the pullback of a line bundle on $B$ so that its restriction to at least one fibre of $\pi$ is not trivial. In other words, $(L_1\otimes L_a)\rvert_{\pi^{-1}(b)} \not\simeq (L_2\otimes L_a^{-1})\rvert_{\pi^{-1}(b)}$ for some point $b \in B$. Define 
\[ F':=(L_1\otimes L_a\otimes \pi^*(\mc{O}_B(b)))\oplus (L_2\otimes L_a^{-1}).\] 
Then, $\det F' = \det E(\pi^{-1}(b))$ and $c_2(F') = c_1(L_1) \cdot c_1(L_2)$ because $c_1(\mathcal{O}_X(\pi^{-1}(b)))$ is a torsion element of $H^2(X,\mathbb{Z})$.
Let $\lambda$ be any line bundle on $\pi^{-1}(b)$ such that $\deg(\lambda)=\lvert Z\rvert > 0$, and define $F$ to be the elementary modification $$ 0 \rightarrow F \rightarrow F' \rightarrow \iota_{b,*}\lambda \rightarrow 0,$$
where $\iota_b$ is the inclusion of $\pi^{-1}(b)$ into $X$. Then $\det(F)=\det(E)$, $c_2(F)=c_2(E)$, and $F$ has maximal destabilising bundles $L_1\otimes L_a$ and $L_2\otimes L_a^{-1}\otimes \pi^*(\mc{O}_B(-b))$. Therefore, $F$ fits into the exact sequences \begin{align*} 0 \rightarrow L_1\otimes L_a \rightarrow &F \rightarrow L_2\otimes L_a^{-1} \otimes \mc{I}_Y \rightarrow 0, \\ 0 \rightarrow L_2\otimes L_a^{-1}\otimes \pi^*(\mc{O}_B(-b)) \rightarrow &F \rightarrow L_1\otimes L_a\otimes \pi^*(\mc{O}_B(b)) \otimes \mc{I}_Y \rightarrow 0\end{align*}
for some finite set of points $Y$ on $X$ (counting multiplicity) with $\lvert Y\rvert = \vert Z \vert >0$.  By assumption, 
\[ \deg(L_2\otimes L_a^{-1} \otimes \pi^*(\mc{O}_B(-b))) < \deg(L_1\otimes L_a)<\deg(L_2\otimes L_a^{-1})\] 
so that
\[ \mu(L_1 \otimes L_a) = \deg(L_1 \otimes L_a) < \frac{1}{2}(\deg L_1 + \deg L_2) = \mu(F)\]
and
\[ \mu(L_2\otimes L_a^{-1}\otimes \pi^*(\mc{O}_B(-b))) = \deg(L_2\otimes L_a^{-1}\otimes \pi^*(\mc{O}_B(-b)))<\frac{1}{2}(\deg L_1 + \deg L_2) = \mu(F).\]  
Hence, $F$ is a stable bundle. Finally, by Proposition \ref{regular}, the trace-free Higgs fields on $F$ are parameterised by $H^0(B, K_B(-b))$ and by Riemann-Roch, $$h^0(B,K_B(-b))=h^0(B,\mc{O}_B(b))+ (g-2) = 1+(g-2) = g-1 \geq 1,$$ 
since $g=g(B) \geq 2$, so $F$ has a non-zero trace-free Higgs field, proving case (i).

Let us now assume that $Z=\emptyset$ and that the spectral curve of $E$ has non-trivial self-intersection. Then, $c_2(E) = c_1(L_1) \cdot c_1(L_2)$ and $c_1(L_1\otimes L_2^{-1})^2\neq 0$. We again choose $L_a$ to be a line bundle on $X$ given by a constant factor of automorphy $a$ such that $$-1<\deg(L_1)-\deg(L_2)+2\deg(L_a)<0.$$
Extensions $$ 0 \rightarrow L_1\otimes L_a \rightarrow F \rightarrow L_2\otimes L_a^{-1}\rightarrow 0$$
are parameterised by 
\[H^1(X,\Hom(L_2\otimes L_a^{-1},L_1\otimes L_a))\cong H^0(B, R^1\pi_*(L_2^{-1}\otimes L_1\otimes L_a^2)),\] 
since $\pi_*(L_2^{-1}\otimes L_1\otimes L_a^2) = 0$, where $R^1\pi_*(L_2^{-1}\otimes L_a\otimes L_a^2)$ is a torsion sheaf supported on points $b \in B$ such that $$\Sigma_{L_1\otimes L_a}\cap \Sigma_{L_2\otimes L_a^{-1}}\cap \pi^{-1}(b)\neq \emptyset.$$
(By the assumption that $c_1(L_2^{-1}\otimes L_1)^2\neq 0$, there is at least one such point in the support.) In addition, for any choice of section $s \in H^0(B, R^1\pi_*(L_2^{-1}\otimes L_1\otimes L_a^2))$, the corresponding extension has maximal destabilising bundles $L_1\otimes L_a$ and $L_2\otimes L_a^{-1}\otimes \pi^*\mc{O}_B(-D_s)$, where $D_s$ is the divisor on which $s$ is supported. Therefore, if we choose $s$ so that it is supported on a single point $b$, the corresponding extension $F$ will have maximal destabilising bundles $L_1\otimes L_a$ and $L_2\otimes L_a^{-1}\otimes \pi^*\mc{O}_B(-b)$. The bundle $F$ is stable and has the same determinant and second Chern class as $E$. Furthermore, the trace-free Higgs fields on $F$ are parameterised by $H^0(B,K_B(-b))\neq 0.$

Finally, let us assume that $Z=\emptyset$ and $\Sigma_1 \cdot \Sigma_2=0$. Therefore, $L_1 \simeq L_2 \otimes L_a \otimes \pi^* H$ for some line bundle $H \in \Pic(B)$ and $a\in \mathbb{C}^*$ with $L_a$ the line bundle with constant automorphy $a$ so that $\det(E)=L_1^2\otimes \pi^*H^{-1}\otimes L_a^{-1}$ and $c_2(E) = c_1(L_1)^2$. Let $W$ be a rank-2 stable bundle on $B$ with determinant $H^{-1}$ and set $F = L_1 \otimes L_b\otimes \pi^*W$, where $b$ is a complex number satisfying $b^2=a^{-1}$. Then, $F$ is a stable bundle on $X$ with $\det(F) = \det(E)$ and $c_2(F) = c_2(E)$. Since $W$ is a stable bundle on $B$, it is simple, implying that $h^0(B,\End_0 V) = 0$. By Riemann-Roch, we then have $h^0(B,\End_0 W \otimes K_B) = h^1(B,\End_0 W) = 3(1-g) > 0$. Hence, $W$ must have a non-trivial trace-free Higgs field, which lifts to a non-trivial trace-free Higgs field on $F$, proving (iii).
\end{proof}

\begin{corollary}\label{cor-filt}
Let $\delta \in \Pic(X)$ and $c_2 \in \mathbb{Z}$ be such that $\Delta(2,\delta,c_2) \geq  m(2,c_1)$. Then, some bundles in $\mc{M}_{2,\delta,c_2}(X)$ admit non-trivial Higgs fields.
\end{corollary}

\begin{remark}
Corollary \ref{cor-filt} tells us that there exist stable rank-2 Higgs bundles on non-K\"ahler elliptic surfaces with non-trivial Higgs field whose underlying bundles are not pulled back from $B$. This is in contrast with what happens on positive elliptic fibrations. Indeed, in \cite{BiswasVerbitsky}, the authors prove that on a positive principal elliptic fibration of dimension $\geq 3$ over a compact K\"ahler orbifold $M$, any stable Higgs bundle is, up to tensoring by a line bundle, the pullback of a Higgs bundle on $M$.
\end{remark}

Finally, we consider the case when the bundles are not filtrable.

\begin{proposition}
Suppose that $g \geq 2$ and $E$ is a rank-2 regular non-filtrable bundle such that $\Delta(E) < \frac{1}{4}(g-1)$. 
Then, $E$ admits non-trivial Higgs fields. In particular, if the spectral curve of $E$ is unramified, then $h^0(X,\End_0 E \otimes K_X) = g-1$.
\end{proposition}

\begin{proof}
Since $E$ is regular and non-filtrable, by Proposition \ref{non-filt-reg}, 
\[h^0(X,\End_0(E) \otimes K_X) = h^0(B,N \otimes K_B),\] 
where $N = \pi_*(\End_0 E)$ and $\deg N = -4\Delta(E)$. Therefore, if $\Delta(E) < \frac{1}{4}(g-1)$, we have $\deg(N \otimes K_B) > g-1$, implying that $h^0(B,N \otimes K_B) \neq 0$.

Suppose that $R= 0$ so that $N^2 \simeq \mathcal{O}_B$. Then, $N = \lambda_0 \neq \mathcal{O}_B$ for some $\lambda_0 \in \Pic^0(B)$ since $h^0(B,N) = h^0(X,\End_0 E) = 0$ by stability of $E$. Moreover, $\lambda_0^2 \simeq \mathcal{O}_B$ and $N^{-1} \simeq N$. Therefore, by Serre Duality and Riemann-Roch,
\[ h^0(X,\End_0 E \otimes K_X) = h^0(B,N \otimes K_B) = h^1(B,N^{-1}) = g-1. \]
\end{proof}

When the Chern classes are chosen in the non-filtrable range, we can say more:

\begin{proposition}\label{non-filt}
Let $\delta \in \Pic(X)$ and $c_2 \in \mathbb{Z}$ be such that $-e_\delta/4 \leq \Delta(2,\delta,c_2) < m(2,c_1)$. Furthermore, suppose that either
\begin{itemize}
    \item $\Delta(2,\delta,c_2)=-e_\delta/4$ and $e_\delta>1-g$ when $g \geq 2$, or
    \item $\Delta(2,\delta,c_2) > -e_\delta/4$ and $e_\delta > 2-g$ when $g \geq 3$.
\end{itemize}
Then, some bundles in $\mc{M}_{2,\delta,c_2}(X)$ admit non-trivial Higgs fields. 
\end{proposition}

\begin{proof}
Set $\Delta = \Delta(2,c_1,c_2)$. Since $\Delta$ is, by assumption, such that $-e_\delta/4 \leq \Delta < m(2,c_1)$, every bundle $E$ with $\Delta(E) = \Delta$ is non-filtrable. If $\Delta = -e_\delta/4$, then $E$ is regular and referring to Proposition \ref{non-filt-reg}, 
\[h^0(X,\End_0(E) \otimes K_X) = h^0(B,N \otimes K_B),\] 
where $N = \pi_*(\End_0 E)$ and $\deg N = -4\Delta = e_\delta$. Therefore, if $e_\delta > 1-g$, we have $\deg(N \otimes K_B) > g-1$, implying that $h^0(B,N \otimes K_B) \neq 0$.
Whereas if $\Delta > -e_\delta/4$, then $\Delta = -e_\delta/4 + k/2$ for some positive integer $k>0$. This means, in particular, that $E$ can have a single jump that has multiplicity $k$ and length 1, in which case $\deg N = e_\delta - 1$ by Corollary \ref{non-filt-jumps} (because the bundle $\bar{E}$ obtained from $E$ by removing its jump is regular with $\Delta(\bar{E}) = -e_\delta/4$ so that $\deg(\pi_*(\End_0 \bar{E})) = e_\delta$ by Proposition \ref{non-filt-reg}).
Therefore, if $e_\delta > 2-g$, we have $\deg(N \otimes K_B) > g-1$ and $h^0(B,N \otimes K_B) \neq 0$. 
\end{proof}

We end the paper with an application of our analysis of Higgs bundles on non-K\"ahler elliptic surfaces to the smoothness of moduli spaces of bundles on them.

\begin{theorem}\label{not-smooth}
Let $\pi:X\to B$ be a non-K\"ahler elliptic surface with base curve $B$ of genus at least 2. Let $\mathcal{M}_{2,\delta,c_2}(X)$ be the moduli space of rank-2 stable bundles on $X$ with determinant $\delta$ and second Chern class $c_2$. Suppose that either
\begin{itemize}
    \item $\Delta(2,c_1,c_2) \geq m(2,c_1)$ when $g \geq 2$, or
    \item $-e_\delta/4 = \Delta(2,\delta,c_2) < m(2,c_1)$ and $e_\delta>1-g$ when $g \geq 2$, or
    \item $-e_\delta/4 < \Delta(2,\delta,c_2) < m(2,c_1)$ and $e_\delta > 2-g$ when $g \geq 3$.
    \end{itemize}
Then, $\mc{M}_{2,\delta,c_2}(X)$ is not smooth as a ringed space when $\Delta(2,\delta,c_2) \neq 0$. In particular, if $c_1(\delta)$ and $c_2$ are in the filtrable range and $\mc{M}_{2,\delta,c_2}(X)$ has positive dimension, then it is not smooth as a ringed space.
\end{theorem}

\begin{remark}
    Note that even if the moduli space is not smooth as a ringed space, the underlying analytic variety may still be smooth. Examples of this phenomenon appear in \cite[Section 3]{Teleman}.
\end{remark}
\begin{proof}
By a standard deformation theory argument, the Zariski tangent space of $\mathcal{M}_{2,\delta,c_2}(X)$ at a point corresponding a vector bundle $E$ is isomorphic to $H^1(X, \End_0(E)).$ Since $E$ is by assumption stable, we have \begin{align*}h^1(X, \End_0(E))&=h^2(X, \End_0(E))-\chi(X, \End_0(E))\\
&=h^0(X, \End_0(E)\otimes K_X)-\chi(X, \End_0(E)),\end{align*}
so $\mathcal{M}_{2,\delta,c_2}(X)$ has a singular point at $E$ if and only if $E$ has a non-zero trace-free Higgs field. The result then follows immediately from Corollary \ref{cor-filt} and Proposition \ref{non-filt}.
\end{proof}


While the results we have presented so far can be used to determine whether a given bundle admits a non-trivial Higgs field, getting a general description for all bundles relies on having an explicit description of the base curve $B$ and of $\Hom(B,T)$. We nonetheless have enough information to completely describe what happens when the genus of $B$ is 2. In this case, $-2 \leq e_\delta \leq 0$.

Suppose that $B$ has genus $g=2$. Let $\iota_B : B \rightarrow \mathbb{P}^1$ be the double cover of $B$ and $p_1,\dots,p_6$ be its ramification points. Let $\delta \in \Pic(B)$ and $c_2 \in \mathbb{Z}$ be such that $-e_\delta/4 \leq \Delta(2,\Delta,c_2) < m(2,c_1)$. Let $E$ be a rank-2 vector bundle with determinant $\delta$ and second Chern class $c_2$. If $R$ is the ramification divisor of the bisection of the spectral curve of $E$ and $N = \pi_*(\End_0 E)$, we have the following:

\begin{example}
Let us first consider the case where $e_\delta = -2$. In this case, $\Delta_0 = \frac{1}{2}$ and $\Delta = \frac{1}{2}(1+k)$, $k \geq 0$. If $k=0$, then the bundles have no jumps and $\deg N =-2$ with $N^{-2} = \mathcal{O}_B(R)$ so that $R$ is then an effective divisor of degree 4. Moreover, $\deg(N \otimes K_B) = 0$ so that $h^0(B,N \otimes K_B) \neq 0$ if and only if $N = K_B^{-1}$, in which case, $\mathcal{O}_B(R) = K_B^2$ and $\iota_B(R) = R$. Note that if $\iota_B(R) \neq R$, then the vector bundle does not admit a Higgs field. Conversely, if $R$ is invariant under the involution, then $\mathcal{O}_B(R) = K_B^2$, implying that $N = K_B^{-1} \otimes \lambda_0$, where $\lambda_0$ is a half-period, and $h^0(B,N \otimes K_B) \neq 0$ if and only if $\lambda_0 = \mathcal{O}_B$. Note that since $\lambda_0$ is a half-period, then $\lambda_0 = \mathcal{O}_B(p_i - p_j)$ for some $1 \leq i,j \leq 6$. Moreover, $K_B = \mathcal{O}_B(2p_i)$ so that
$N = \mathcal{O}_B(-p_i-p_j)$, and $h^0(B,N \otimes K_B) \neq 0$ if and only if $i=j$.

Let us now assume that $k > 0$. If the vector bundle has no jumps, then $\deg N = -2(k+1)$ and $\deg(N \otimes K_B) = -2k < 0$. Therefore, $h^0(B,N \otimes K_B) = 0$, implying that there are no non-trivial Higgs fields in this case. 
Let us finally consider the case where $k > 0$ and the bundle has jumps. In this case, $\deg N = -2 -(2m + s)$ with $0 \leq m \leq k$ and $s \geq 1$ so that $\deg(N \otimes K_B) = -(2m+s) < 0$ and $h^0(B,N \otimes K_B) = 0$. Hence, there are no non-trivial Higgs fields when $k>0$. 

We see that when $e_\delta = -2$ and $k>0$, there are no non-trivial Higgs fields, implying that the moduli spaces are always smooth in this case. 
\end{example}

\begin{example}
Let us now consider the case where $e_\delta = -1$. In this case, $\Delta_0 = \frac{1}{4}$ and $\Delta = \frac{1}{4}(1+2k)$, $k \geq 0$. If $k=0$, then the bundles have no jumps and $\deg N =-1$ with $N^{-2} = \mathcal{O}_B(R)$ so that $R$ is then an effective divisor of degree 2. Moreover, $\deg(N \otimes K_B) = 1$ and $h^0(B,N \otimes K_B) \neq 0$ if and only if $h^1(B,N^{-1}) \neq 0$, in which case, $N^{-1} = \mathcal{O}_B(b)$ for some $b \in B$. Moreover, $\mathcal{O}_B(R) = N^{-2} = \mathcal{O}_B(2b)$. But since $R$ is an effective divisor of degree 2 without double points, it can only be linearly equivalent to the double point $2b$ if $b=p_i$ for some $1 \leq i \leq 6$ (because, otherwise, the linear system of $R$ is 0-dimensional) so that $R \sim 2p_i \sim K_B$ and $\mathcal{O}_B(R) = K_B$. Note that if $R \not\sim K_B$, then the vector bundle does not admit a Higgs field. Conversely, if $R \sim K_B$, then $R \sim 2p_i$ for any $1 \leq i \leq 6$, implying that $N^{-1} = \mathcal{O}(p_i) \otimes \lambda_0$, where $\lambda_0$ is a half-period. Again, $h^0(B,N \otimes K_B) \neq 0$ if and only if $N^{-1} = \mathcal{O}_B(b)$ for some $b \in B$, in which case $\mathcal{O}_B(b-p_i) \simeq \lambda_0$ and $\mathcal{O}_B(2b-2p_i) = \mathcal{O}_B$. But this means that $2b \sim 2p_i \sim K_B$ so that $b=p_j$ for some $1 \leq j \leq 6$. We therefore see that $h^0(B,N \otimes K_B) \neq 0$ if and only if $R \sim K_B$ and $N = \mathcal{O}_B(-p_j)$ for some $1 \leq j \leq 6$.

Let us now assume that $k > 0$. If the vector bundle has no jumps, then $\deg N = -(2k+1)$ and $\deg(N \otimes K_B) = 1-2k < 0$. Therefore, $h^0(B,N \otimes K_B) = 0$, implying that there are no non-trivial Higgs fields in this case. 
Let us finally consider the case where $k > 0$ and the bundle has jumps. In this case, $\deg N = -1 -(2m + s)$ with $0 \leq m \leq k$ and $s \geq 1$ so that $\deg(N \otimes K_B) = 1-(2m+s) \leq 0$. If $\deg(N \otimes K_B) = 0$, then $2m+s = 1$ so that $m=0$ and $s=1$, implying that the bundle has only one jump with length dividing $k$ and each allowable elementary modification in the sequence preserves the height except for the final one. 
Then, $\deg N = -2$ so that $N^{-1} = \mathcal{O}_B(p+p')$ for some $p,p' \in B$. Moreover, $h^0(B,N \otimes K_B) \neq 0$ if and only if $N = K_B^{-1}$ if and only if $p' = \iota_B(p)$. Also, $N' = N(b)$, where $b \in B$ is such that the jump occurs over the fibre $\pi^{-1}(b)$. Note that $K_B \simeq \mathcal{O}(b+\iota(b))$. Therefore, if $N = K_B^{-1}$, then $N' = K_B^{-1}(b)$ and $\mathcal{O}_B(R) = K_B^2(-2b) \simeq \mathcal{O}_B(2\iota_B(b))$. But $R$ is an effective divisor of degree 2. So the only way that $R \sim 2\iota_B(b)$ is if $\iota_B(b) = p_i$ for some $1 \leq i \leq 6$. Hence, $R \sim K_B$ so that $R = p + \iota_B(p)$ with $\iota_B(p) \neq p$, and $b$ is a ramification point of $B$. In particular, if $R \not\sim K_B$ or if the fibre of the jump is not over a ramification point of $B$, then there are no non-trivial Higgs fields.
Conversely, if the fibre of the jump is over the ramification point $p_i$ and $R = p + \iota_B(p)$ with $\iota_B(p) \neq p$, then $N^{-2}(-2p_i) = \mathcal{O}_B(R) \simeq K_B$. But $\mathcal{O}(2p_i) \simeq K_B$ so that $N^{-2} = K_B^2$ and $N = K_B^{-1} \otimes \lambda_0$ for some half-period $\lambda_0$. Suppose that $\lambda_0 = \mathcal{O}_B(p_j - p_k)$ for some $1 \leq j,k \leq 6$. Then, since $K_B \simeq \mathcal{O}_B(2p_j)$, we see that $N=\mathcal{O}_B(-p_j-p_k)$ and 
$h^0(B,N \otimes K_B) \neq 0$ if and only if $j=k$. 

We see that when $e_\delta = -1$ and $k>0$, there are non-trivial Higgs fields only when $R \simeq K_B \simeq N^{-2}$ and the bundle has a unique jump of dividing $k$ over a ramification point of $B$.  
\end{example}

\begin{example}
We finally consider the case where $e_\delta = 0$. This time, $\Delta_0 = 0$ and $\Delta = \frac{k}{2}$, $k \geq 0$. If $k=0$, then the bundles have no jumps and $\deg N =0$ with $N^{-2} = \mathcal{O}_B(R)$ so that $R=0$ since $R$ is then an effective divisor of degree 0. This means that the spectral covers of the bundles are unramified. Moreover, $N$ is a non-trivial half-period so that $N = \mathcal{O}_B(p_i - p_j)$ with $i \neq j$. Moreover, $\deg(N \otimes K_B) = 2$ so that $h^0(B,N \otimes K_B) \neq 0$. Therefore, any vector bundle with $\Delta = 0$ admits a non-trivial Higgs field.

Let us now assume that $k > 0$. If the vector bundle has no jumps, then $\deg N = -2k$ and $\deg(N \otimes K_B) = 2-2k \leq 0$. In particular, $\deg(N \otimes K_B) < 0$ when $k > 1$ so that $h^0(B,N \otimes K_B) = 0$, implying that there are no non-trivial Higgs fields in this case. But when $k=1$, we have that $\deg(N \otimes K_B) = 0$ so that $h^0(B,N \otimes K_B) \neq 0$ if and only if $N = K_B^{-1}$, in which case, $\mathcal{O}_B(R) = K_B^2$ so that $\iota_B(R) = R$. In other words, if the bundle admits a non-trivial Higgs field, then the ramification divisor $R$  of its spectral curve is invariant under the involution $\iota_B$ on $B$. This means, in particular, that if the ramification divisor of the spectral cover is not invariant under the involution $\iota_B$, then the bundle does not admits non-trivial Higgs fields. Conversely, if $\iota_B(R) = R$, then $R = K_B^2$ since $R$ is an effective divisor of degree $8\Delta = 4$. Therefore, $N^{-2} = K_B^2$, implying that $N = K_B^{-1} \otimes \lambda_0$ for some half-period $\lambda_0 \in \Pic^0(B)$, and $h^0(B,N \otimes K_B) \neq 0$ if and only if $\lambda_0 = \mathcal{O}_B$.

We end with the case where $k > 0$ and the bundle has jumps. In this case, $\deg N = -(2m + s)$ with $0 \leq m \leq k$ and $s \geq 1$ so that $\deg(N \otimes K_B) = 2-(2m+s) \leq 1$. If $m$ and $s$ are such that $\deg(N \otimes K_B) < 0$, then $h^0(B,N \otimes K_B) = 0$ and there are no non-trivial Higgs fields. We therefore only have to consider the case where $\deg(N \otimes K_B) = 0,1$. We first note that $\deg(N \otimes K_B) = 1$ if and only if $m=0$ and $s=1$, implying that the bundle has only one jump with length dividing $k$ and each allowable elementary modification in the sequence preserves the height except for the final one. Moreover, $N = N'(-b)$, where $b \in B$ is such that the jump occurs over the fibre $\pi^{-1}(b)$ and $N'$ is a non-trivial half-period on $B$. In particular, $\mathcal{O}_B(R) = (N')^{-2} = \mathcal{O}_B$ so that $R = 0$ and the spectral curve of the bundle is unramified. If $h^0(B,N \otimes K_B) \neq 0$, then $h^0(B,N^{-1}) = h^1(B,N^{-1}) = h^0(B,N \otimes K_B) \neq 0$ and so $N^{-1} = \mathcal{O}(b')$ for some $b' \in B$. Thus, $N' = \mathcal{O}_B(b-b')$ so that $b' \neq b$. Conversely, if the spectral cover is unramified so that $R = 0$ and $N' = \mathcal{O}_B(p_i - p_j)$, $i \neq j$, we can add a jump at $p_i$, which will give $N = \mathcal{O}_B(-p_j)$ so that $N \otimes K_B = \mathcal{O}_B(p_j)$ since $p_j$ is a ramification points, implying that $h^0(B,N \otimes K_B) \neq 0$.

Finally, if $\deg(N \otimes K_B) = 0$, then $2m+s = 2$ so that $m=0$ and $s=2$, implying that the bundle has 2 jumps. Again, $R = 0$ and $N'$ is a non-trivial half-period. Then, $h^0(B,N \otimes K_B) \neq 0$ if and only if $N = K_B^{-1}$. Moreover, $N = N'(-b-b')$ where $b,b' \in B$ corresponds to the fibres over which the jumps occur. Hence, $N' = K_B^{-1} \otimes \mathcal{O}_B(b+b')$, implying that $\mathcal{O}_B(b+b') \neq K_B$, which means that $b' \neq \iota_B(b)$. In other words, the jumps correspond to fibres over points in $B$ that are not related by the involution. Conversely, if $R = 0$ and $N' = \mathcal{O}_B(p_i - p_j)$, $i \neq j$, then if we introduce jumps at $p_i$ and $p_j$, the resulting bundle will have $N = \mathcal{O}_B(-2p_j) = K_B^{-1}$, implying that $h^0(B,N \otimes K_B) \neq 0$.

We see that when $e_\delta = 0$, there always exist bundle that admits non-trivial Higgs fields, implying that the moduli spaces are always singular in this case. 
\end{example}

\begin{remark}
Referring to Theorem 3 in \cite{Kani}, minimal genus 2 covers of $T$ of degree $d \geq 2$ always exist, implying that there exists a principal elliptic surfaces $\pi:X\to B$ with $g(B)=2$ and a line bundle $\delta \in \Pic(X)$ such that $m(2,\delta)=d/4$ for any such $d$.
\end{remark}

As a direct consequence of the above examples, we obtain:

\begin{theorem}
Suppose that $B$ has genus $g=2$. Let $\delta \in \Pic(B)$, $c_2 \in \mathbb{Z}$, and $\mathcal{M}_{2,\delta,c_2}(X)$ be the moduli space of rank-2 stable bundles on $X$ with determinant $\delta$ and second Chern class $c_2$.
Suppose that $-e_\delta/4 \leq \Delta(2,\Delta,c_2) < m(2,c_1)$ so that every bundle in $\mathcal{M}_{2,\delta,c_2}(X)$ is non-filtrable. Then:
\begin{enumerate}
    \item If $e_\delta = -2$ and $\Delta(2,\delta,c_2) > 1/2$, then $\mc{M}_{2,\delta,c_2}(X)$ is smooth. 
    \item  If $e_\delta = 0$, then $\mc{M}_{2,\delta,c_2}(X)$ is singular as a ringed space.
\end{enumerate}

\end{theorem}

\bibliographystyle{alpha}
\bibliography{ref}
\end{document}